\documentclass[notitlepage, pra, 10pt, aps]{revtex4-2}
\usepackage{amsmath,amssymb,amsfonts}
\usepackage{graphicx}
\usepackage{bm}
\usepackage{psfrag}
\usepackage[caption=false]{subfig}
\usepackage{color}
\usepackage{xcolor}
\usepackage{mathtools}
\usepackage{hyperref}
\usepackage[makeroom]{cancel}

\DeclareMathOperator*{\argmin}{arg\,min}

\newtheorem{theorem}{Theorem}
\newtheorem{lemma}{Lemma}

\newtheorem{proof}{Proof}
\newtheorem{assumption}{Assumption}

\allowdisplaybreaks[1]

\begin{document}

\title{Distributed Computing for Huge-Scale Aggregative Convex Programming}

\author{Luoyi Tao}
\email[]{luoyitao@smail.iitm.ac.in, taoluoyi@gmail.com}
\affiliation{Geophysical Flows Lab\\
                and\\
             Department of Aerospace Engineering\\
             Indian Institute of Technology Madras\\
             Chennai 600 036, India}

\date{\today}

\def\s{\!}
\def\ss{\!\!}
\def\sss{\!\!\!}
\def\l{\left}
\def\r{\right}

\def\Prox{\text{Prox}}

\def\IndicatorFunction{\delta}
\def\ConeK{{\cal K}}
\def\SubjectTo{\text{subject to}}

\def\cA{{\cal A}}
\def\cB{{\cal B}}
\def\cC{{\cal C}}
\def\cD{{\cal D}}
\def\cE{{\cal E}}
\def\cF{{\cal F}}
\def\cG{{\cal G}}
\def\cH{{\cal H}}
\def\cI{{\cal i}}
\def\cK{{\cal K}}
\def\cL{{\cal L}}
\def\cP{{\cal P}}
\def\cQ{{\cal Q}}
\def\cR{{\cal R}}
\def\dist{\textbf{dist}}
\def\dom{\textbf{dom}}
\def\Indicator{I}
\def\maximize{\text{maximize}}
\def\minimize{\text{minimize}}
\def\overX{\bar{X}}
\def\overY{\bar{Y}}
\def\Projection{\Pi}
\def\prox{\textbf{prox}}
\def\subjectto{\text{subject to}}
\def\Lik{\Lambda_i^k}
\def\Ljk{\Lambda_j^k}
\def\NCV{N_{\text{CV}}}

\def\LDmu{\mu}
\def\LDnu{\nu}
\def\LDbeta{\beta}
\def\pcE{E}
\def\pcF{F}
\def\pcG{G}
\def\pcH{H}
\def\svG{\prescript{G\!}{}{Y\!}}
\def\dvG{\prescript{G\!}{}{\!\mu}}
\def\tdvG{\prescript{G}{}{\!\tilde{\mu}}}
\def\svGcoef{\prescript{G\!}{}{\!\gamma}}
\def\augGcoef{\prescript{G\!}{}{\!\rho}}

\def\dvZ{\prescript{X}{}{\!\mu}} 
\def\dvZp{\prescript{p\!X}{}{\!\mu}}
\def\tdvZp{\prescript{p\!X}{}{\!\tilde{\mu}}}
\def\dvZn{\prescript{n\!X}{}{\!\mu}}
\def\tdvZn{\prescript{n\!X}{}{\!\tilde{\mu}}}
\def\svZ{Y\!}
\def\svZp{\prescript{p\!X}{}{\!\svZ}}
\def\svZn{\prescript{n\!X}{}{\!\svZ}}
\def\svZpcoef{\prescript{p\!X}{}{\!\gamma}}
\def\svZncoef{\prescript{n\!X}{}{\!\gamma}}

\def\dvH{\prescript{H}{}{\!\!\mu}} 
\def\dvHp{\prescript{p\!H}{}{\!\!\mu}}
\def\tdvHp{\prescript{p\!H}{}{\!\tilde{\mu}}}
\def\dvHn{\prescript{n\!H}{}{\!\!\mu}}
\def\tdvHn{\prescript{n\!H}{}{\!\tilde{\mu}}}
\def\svH{Y\!}
\def\svHp{\prescript{p\!H}{}{\!\svH}}
\def\svHn{\prescript{n\!H}{}{\!\svH}}
\def\svHpcoef{\prescript{p\!H}{}{\!\gamma}}
\def\svHncoef{\prescript{n\!H}{}{\!\gamma}}

\def\svE{\prescript{E\!}{}{Y\!}}
\def\dvE{\prescript{E\!}{}{\!\mu}}
\def\tdvE{\prescript{E}{}{\!\tilde{\mu}}}
\def\svEcoef{\prescript{E\!}{}{\!\gamma}}
\def\augEcoef{\prescript{E\!}{}{\!\rho}}
\def\JE{\prescript{E\!\!}{}{J}}

\def\svF{\prescript{F\!}{}{Y\!}}
\def\dvF{\prescript{F\!}{}{\!\mu}}
\def\tdvF{\prescript{F}{}{\!\tilde{\mu}}}
\def\svFcoef{\prescript{F\!}{}{\!\gamma}}
\def\augFcoef{\prescript{F\!}{}{\!\rho}}
\def\JF{\prescript{F\!\!\!}{}{J}}
\def\JFo{\prescript{F_{1}\!\!\!}{}{J}}
\def\JFt{\prescript{F_{2}\!\!\!}{}{J}}

\def\proxXcoef{\sigma}
\def\proxXNormOcoef{\prescript{1}{}{\!\sigma}}
\def\proxXNormTcoef{\prescript{2}{}{\!\sigma}}
\def\svGcoef{\prescript{G\!}{}{\!\gamma}}

\def\dvEcoef{\prescript{E\!}{}{\alpha}}
\def\dvFcoef{\prescript{F\!}{}{\alpha}}
\def\dvGcoef{\prescript{G\!}{}{\alpha}}
\def\dvZpcoef{\prescript{p\!X}{}{\!\alpha}}
\def\dvZncoef{\prescript{n\!X}{}{\!\alpha}}
\def\dvHpcoef{\prescript{p\!H}{}{\!\alpha}}
\def\dvHncoef{\prescript{n\!H}{}{\!\alpha}}
\def\GammaE{\prescript{E\!}{}{\Gamma}}
\def\GammaF{\prescript{F\!}{}{\Gamma}}

\def\eE{\prescript{E\!}{}{e}}
\def\eF{\prescript{F\!}{}{e}}
\def\eG{\prescript{G\!}{}{e}}
\def\eZp{\prescript{p\!X}{}{\!e}}
\def\eZn{\prescript{n\!X}{}{\!e}}
\def\eHp{\prescript{p\!H}{}{\!e}}
\def\eHn{\prescript{n\!H}{}{\!e}}

\def\udvE{\prescript{E\!}{}{u}}
\def\udvF{\prescript{F\!}{}{u}}
\def\udvG{\prescript{G\!}{}{u}}
\def\udvZp{\prescript{p\!X}{}{\!u}}
\def\udvZn{\prescript{n\!X}{}{\!u}}
\def\udvHp{\prescript{p\!H}{}{\!u}}
\def\udvHn{\prescript{n\!H}{}{\!u}}

\def\usvE{\prescript{E\!}{}{u}}
\def\usvF{\prescript{F\!}{}{u}}
\def\usvG{\prescript{G\!}{}{u}}
\def\usvZp{\prescript{p\!X}{}{\!u}}
\def\usvZn{\prescript{n\!X}{}{\!u}}
\def\usvHp{\prescript{p\!H}{}{\!u}}
\def\usvHn{\prescript{n\!H}{}{\!u}}

\def\proxXpropto{\prescript{X}{}{\!\kappa}}
\def\proxZpropto{\prescript{Z}{}{\!\kappa}}
\def\proxsvpXpropto{\prescript{p\!X}{}{\!\kappa}}
\def\proxsvnXpropto{\prescript{n\!X}{}{\!\kappa}}
\def\proxsvFpropto{\prescript{F}{}{\!\kappa}}
\def\proxsvGpropto{\prescript{G}{}{\!\kappa}}
\def\proxsvpHpropto{\prescript{p\!H}{}{\!\kappa}}
\def\proxsvnHpropto{\prescript{n\!H}{}{\!\kappa}}

\def\proxXproptorhoi{\prescript{X}{}{\!\lambda_{i,l}}}
\def\proxZproptorhoi{\prescript{Z}{}{\!\lambda}}
\def\proxsvpXproptorhoi{\prescript{p\!X}{}{\!\lambda_{i}}}
\def\proxsvnXproptorhoi{\prescript{n\!X}{}{\!\lambda_{i}}}
\def\proxsvFproptorhoi{\prescript{F}{}{\!\lambda_{i}}}
\def\proxsvGproptorhoi{\prescript{G}{}{\!\lambda_{i}}}
\def\proxsvpHproptorhoi{\prescript{p\!H}{}{\!\lambda_{i}}}
\def\proxsvnHproptorhoi{\prescript{n\!H}{}{\!\lambda_{i}}}

\def\udvpXpropto{\prescript{p\!X}{}{\!\!\beta}}
\def\udvnXpropto{\prescript{n\!X}{}{\!\!\beta}}
\def\udvFpropto{\prescript{F}{}{\!\!\beta}}
\def\udvGpropto{\prescript{G}{}{\!\!\beta}}
\def\udvpHpropto{\prescript{p\!H}{}{\!\!\beta}}
\def\udvnHpropto{\prescript{n\!H}{}{\!\!\beta}}

\def\phiF{\prescript{F\!}{}{\phi}}

\def\proj{P}

\def\Xy{\prescript{1\!}{}{X\!}}
\def\Xr{\prescript{2\!}{}{X\!}}
\def\Xs{\prescript{3\!}{}{X\!}}
\def\Zy{\prescript{1\!}{}{Z\!}}
\def\Zr{\prescript{2\!}{}{Z\!}}
\def\Zs{\prescript{3\!}{}{Z\!}}
\def\uy{\prescript{1\!}{}{u\!}}
\def\ur{\prescript{2\!}{}{u\!}}
\def\us{\prescript{3\!}{}{u\!}}
\def\Ny{\prescript{1\!}{}{N\!}}
\def\Nr{\prescript{2\!}{}{N\!}}
\def\Ns{\prescript{3\!}{}{N\!}}
\def\proxYNormOcoef{\prescript{1\!}{}{\gamma\!}}

\begin{abstract}
Concerning huge-scale aggregative convex programming of a linear objective subject to the affine constraints of equality and inequality and the quadratic constraints of inequality, convex and aggregatively computable, an algorithm is developed for its distributed computing. The consensus with single common variable is used to partition the constraints into multi-consensus blocks, and the subblocks of each consensus block are employed to partition the primal variables into multiple sets of disjoint subvectors. The global consensus constraints of equality and the original constraints are converted into the extended constraints of equality via slack variables to help initialize the algorithm. The augmented Lagrangian, the proximal point method with double proximal terms or single, the block-coordinate Gauss-Seidel method, and ADMM are used to update the primal and slack variable sequences; the descent models with built-in bounds are used to update the dual, motivated by the mathematical structures of the first-order characteristics of the update rules for the primal and slack. The feasibility conditions for the algorithm to produce optimal solutions are described and their realizations through precondition and parameter and initial values are outlined or specified. Under the feasibility conditions supposed, convergence of the algorithm to optimal solutions is argued and the rate of convergence, $O(1/k^{1/2})$ is estimated roughly. Issues requiring further explorations are listed.

\begin{description}
\item[Mathematics Subject Classification]
\verb+90C05+
\verb+90C06+
\verb+90C25+
\verb+90C30+
\verb+68W15+
\verb+68W40+
\end{description}
\end{abstract}

\maketitle

\section{Introduction}

Toward the exploration of hydrodynamic turbulence modeling
within the framework of optimal correlation functions \cite{Tao2020},
an algorithm for distributed computing of
linear programming of huge-scales is proposed in \cite{Tao2025v7LP}.
In this work, we extend the linear program by including
 the quadratic constraints of inequality
 that are convex and aggregatively computable
and develop a distributed computing algorithm
to solve this extended problem.

Specifically, the study 
employs the techniques outlined in the following, either conventional or new.
(a) The global consensus with single common variable is applied to
partition the constraints of equality and inequality into multiple consensus blocks;
the subblocks of each consensus block partition the primal variables
into multiple sets of disjoint subvectors
 to make computation of the primal variables feasible.
(b)
 The global consensus constraints of equality and the original constraints
 of equality and inequality are converted
 into the extended constraints of equality via slack variables, the latter
 treated as independent,
so as to help resolve the issues of
 initialization and feasibility of the algorithm.
(c)
The augmented Lagrangian,
the proximal point method, the block-coordinate Gauss-Seidel method, and ADMM are
 used to update the primal and slack sequences iteratively.
(d)
The primal sequence in a subblock is updated
with the help of double proximal terms,
one involving the 2-norm and the other the 1-norm,
in order to help make the algorithm feasible.
Motivated by the mathematical structures of the first-order characteristics
of the update rules for the primal and slack sequences,
the descent models with built-in bounds are proposed
for the updates of the dual sequences.
(e)
The objective and the constraints are preconditioned.
The feasibility conditions for the algorithm to produce optimal solutions
are described and the assumptions underlying the analysis
are listed with some justifications.
(f)
The initial values of the primal, slack, and dual sequences,
the values of the parameters,
and the upper bounds of the slack and dual
are outlined or specified
in order to realize the feasibility conditions.
(g)
Convergence of the algorithm to optimal solution is shown
 and the rate of convergence, $O(1/k^{1/2})$ is roughly estimated
 under the feasibility conditions supposedly satisfied.
(h)
The issues yet to be resolved further are listed.

The paper is organized as follows.
Section~\ref{sec:Formulation} formulates the optimization problem.
Section~\ref{sec:UpdatePrimal} presents the methodology for updates of
the primal and slack sequences.
Section~\ref{sec:UpdateDual} discusses the precondition of the domain,
the objective, and the constraints,
the descent updates for the dual sequences,
and the feasibility conditions; it also presents the analysis of convergence of the algorithm.
The issues regarding initialization
and estimation of the parameters and the upper bounds for the slack and dual
are discussed in more details in Section~\ref{sec:Initialization}.
Section~\ref{sec:Summary} summarizes the main results and lists some issues
that need to be explored further.

\section{Formulation of optimization problem}
\label{sec:Formulation}
The primal problem is aggregative convex programming,
\begin{align}
\label{ACPGSPrimalProblemOriginal}
&
\minimize\ \ \  f(Z)
\notag\\[-6.5pt]&
\\[-6.5pt]&
\subjectto\ \ Z\in\cC.
\notag
\end{align}
Here, $f(Z)$ is linear;
\begin{align}
 \cC=\{Z\in \mathbb{R}^n: 
 \pcF(Z)\leq 0,\ 
 \pcG(Z)\leq 0,\ 
 \pcH(Z)=0,\ 
 l\leq Z\leq u\},
\label{ACPGSPrimalProblemConstraintSeto}
\end{align}
$\pcF$ is a vector function whose components are convex, quadratic,
and aggregatively computable, such as
\begin{align}
e(Z)=a(Z)+c(Z)^2,\ \
e(Z)=c(Z)^2-a(Z) b(Z), 
\label{NLGSQuadraticConvexConstraints}
\end{align}
where $a(Z)$, $b(Z)$, and $c(Z)$ are affine,
$\pcG$ and $\pcH$ are vector functions whose components are affine,
$l$ and $u$ are the lower and upper bounds for $Z$, respectively.
Without loss of generality, we take $l=-u$,
which can be obtained through the translational shift of the domain,
$f$ is treated equivalently as linear,
 $a(Z)$, $b(Z)$, $c(Z)$,
$\pcG(Z)$, and $\pcH(Z)$ as affine in the shifted domain.

Considering the huge-scale computational size,
both the value of $n$ and the number of constraints contained 
in \eqref{ACPGSPrimalProblemConstraintSeto} being great,
we first partition \eqref{ACPGSPrimalProblemOriginal} constraint-wise
via global consensus with single common variable $Z$ 
\cite{Boydetal2010, ParikhBoyd2013},
\begin{align}
\label{ACPGSPrimalProblemGlobalConsensusRing}
&
 \minimize\ \ \, \sum_{i=1}^N \big(f(X_i) +\Indicator_{\cC_i}(X_i)\big)
\notag\\[-13pt]&
\\[-2pt]&
\subjectto\ \ X_i-Z=0, \ \  i=1,\ldots,N.
\notag
\end{align}
Here, $\cC$ is partitioned into $N$ blocks, $\cC=\cup_{i=1}^{N}\cC_i$
with $\cC_i$ given by 
\begin{align}
 \cC_i=\big\{X_i\in [-u, u]:
    \pcF_{i}(X_i)\leq 0,\ 
    \pcG_i(X_i)\leq 0,\ 
    \pcH_i(X_i)=0\big\},
\label{ACPGSConstraintsithCBo}
\end{align}
and $\Indicator_{\cC_i}$ is the indicator function of $\cC_i$. 
Next, to have appropriate initialization of the algorithm, 
$\{X_i-Z=0\}$ is replaced by $\{X_{i}-Z\leq 0$, $Z-X_i\leq 0\}$
and $\{\pcH_i(X_i)=0\}$ by $\{\pcH_i(X_i)\leq 0$, $-\pcH_i(X_i)\leq 0\}$ equivalently; 
 $\cC_i$ of \eqref{ACPGSConstraintsithCBo} is extended to
\begin{align}
 \cC_i=\big\{X_i\in [-u, u]:
        X_i-Z\leq 0,
        Z-X_{i}\leq 0,
        Z\in [-u,u],  \,
        \pcF_{i}(X_i)\leq 0,
        \pcG_i(X_i)\leq 0,
        \pcH_i(X_i)\leq 0,
        -\pcH_i(X_i)\leq 0
        \big\}.
\label{ACPGSConstraintsithCB}
\end{align}
Further,
to help initialize the algorithm, the slack variables,
$\{\svZp_{i}, \svZn_{i}, \svF_{i}, \svG_{i},
    \svHp_{i}, \svHn_{i}\}$
are introduced
 to convert all the constraints of inequality in $\cC_i$ 
of \eqref{ACPGSConstraintsithCB}
into the extended constraints of equality,
 following the conventional practice,
\begin{align}
\cC_i=\big\{&
        X_i\in [-u, u]:\
        X_{i}-Z+\svZp_{i}=0,\
        Z-X_{i}+\svZn_{i}=0,\ 
        \svZp_{i}\in[0, \usvZp_{Y_{i}}],\ 
        \svZn_{i}\in[0, \usvZn_{Y_{i}}],\ 
        Z\in [-u,u],\
        \notag\\&\hskip10mm
        \pcF_i(X_i)+\svF_i= 0,\ 
        \pcG_i(X_i)+\svG_i= 0,\
        \pcH_i(X_i)+\svHp_{i}=0,\
        -\pcH_i(X_i)+\svHn_{i}=0,\
        \svF_{i}\in[0, \usvF_{Y_{i}}],
        \notag\\&\hskip10mm
        \svG_{i}\in[0, \usvG_{Y_i}],\
        \svHp_{i}\in[0, \usvHp_{Y_{i}}],\ 
        \svHn_{i}\in[0, \usvHn_{Y_{i}}]
    \big\}.
\label{ACPGSConstraintsithCBSlacked}
\end{align}
How to set the upper bounds for the slack variables,
$\{\usvZp_{Y_{i}}, \usvZn_{Y_{i}}, \usvF_{Y_i}, \usvG_{Y_{i}},
    \usvHp_{Y_{i}}, \usvHn_{Y_{i}}\}$
is to be discussed later.
For convenience, we call this $\cC_i$ the $i$-th consensus block (CB)
or the $i$-CB to indicate the relevant operations involved. 
The above treatment involves the conversion of the global consensus constraints of equality
in \eqref{ACPGSPrimalProblemGlobalConsensusRing}
and the constraints of equality $\{\pcH_i(X_i)=0\}$ in \eqref{ACPGSConstraintsithCBo}
to the equivalent inequality constraints in \eqref{ACPGSConstraintsithCB}
and further to the extended constraints of equality involving slack variables
in \eqref{ACPGSConstraintsithCBSlacked}, following \cite{Tao2025v7LP}.
Though increasing the computational size,
this conversion plays a significant role to ensure the feasibility
and adequate initialization of the algorithm
proposed, as to become clear.
Hereafter, the slack variables are viewed as part of the primal variables.

To help solve the objective function of \eqref{ACPGSPrimalProblemGlobalConsensusRing}
subject to \eqref{ACPGSConstraintsithCBSlacked},
 we employ the augmented Lagrangian $L$,
\begin{align}
&
 L\big(X,Z,\svZp,\svZn,\svF,\svG,\svHp,\svHn,
            \dvZp,\dvZn,\dvF,\dvG,\dvHp,\dvHn,\rho\big)
\notag\\[-3pt]
=\,&
 \sum_{i=1}^N L_i\big(X_i,Z,\svZp_i,\svZn_i,\svF_i,\svG_i,\svHp_i,\svHn_i,
    \dvZp_i,\dvZn_i,\dvF_i,\dvG_i,\dvHp_i,\dvHn_i,\rho_i\big).
\label{ACPGSLagranginDualConsensus}
\end{align}
Here,
\begin{align}
&
  L_i\big(X_i,Z,\svZp_i,\svZn_i,\svF_i,\svG_i,\svHp_i,\svHn_i,
    \dvZp_i,\dvZn_i,\dvF_i,\dvG_i,\dvHp_i,\dvHn_i,\rho_i\big)
\notag\\
=\,&
f(X_{i})
    +\big\langle  \dvZp_{i}, X_{i}-Z+\svZp_{i}\big\rangle
    +\big\langle  \dvZn_{i}, Z-X_{i}+\svZn_{i}\big\rangle
    +\big\langle  \dvF_{i}, \pcF_{i}(X_{i})+\svF_{i}\big\rangle
    +\big\langle  \dvG_{i},  \pcG_{i}(X_{i})+\svG_{i}\big\rangle
    \notag\\&
    +\big\langle  \dvHp_{i}, \pcH_{i}(X_{i})+\svHp_{i}\big\rangle
    +\big\langle  \dvHn_{i},-\pcH_{i}(X_{i})+\svHn_{i}\big\rangle
    +\frac{\rho_i}{2}\Big(
            \big\Vert X_{i}-Z+\svZp_{i} \big\Vert^2
            +\big\Vert Z-X_{i}+\svZn_{i} \big\Vert^2
            \notag\\&\hskip30mm
            +\big\Vert \pcF_{i}(X_{i})+\svF_{i} \big\Vert^2
            +\big\Vert \pcG_{i}(X_{i})+\svG_{i} \big\Vert^2
            +\big\Vert \pcH_{i}(X_{i})+\svHp_{i} \big\Vert^2
            +\big\Vert -\pcH_{i}(X_{i})+\svHn_{i} \big\Vert^2
    \Big),
\label{ACPGSLagranginFunctionithBC}
\end{align} 
which is the augmented Lagrangian function for the $i$-CB,
$\rho_i$ is the positive penalty parameter,
$\dvZp_{i}$,
$\dvZn_{i}$,
$\dvF_{i}$,
$\dvG_{i}$,
$\dvHp_{i}$,
and $\dvHn_{i}$ are
the dual variables associated with
$X_{i}-Z_{}+\svZp_{i}=0$,
$Z_{}-X_{i}+\svZn_{i}=0$,  
$\pcF_{i}(X^{}_{i})+\svF_{i}=0$,
$\pcG_{i}(X^{}_{i})+\svG_{i}=0$,
$\pcH_{i}(X_{i})+\svHp_{i}=0$,
and $-\pcH_{i}(X_{i})+\svHn_{i}=0$,
respectively.
The dual variables are supposedly finite
and their bounds are to be specified.
\begin{align}
&
X=(X_1,\ldots,X_N),\ 
\svZp=(\svZp_1,\ldots, \svZp_N),\
\svZn=(\svZn_1,\ldots, \svZn_N),\
\svF=(\svF_1,\ldots, \svF_N),\
\svG=(\svG_1,\ldots, \svG_N),
\notag\\&
\svHp=(\svHp_1,\ldots, \svHp_N),\
\svHn=(\svHn_1,\ldots, \svHn_N),\
\dvZp=(\dvZp_1,\ldots, \dvZp_N),\
\dvZn=(\dvZn_1,\ldots, \dvZn_N),
\notag\\&
\dvF=(\dvF_1,\ldots, \dvF_N),\
\dvG=(\dvG_1,\ldots, \dvG_N),\
\dvHp=(\dvHp_1,\ldots, \dvHp_N),\
\dvHn=(\dvHn_1,\ldots, \dvHn_N),\
\rho=(\rho_1,\ldots, \rho_N),
\end{align}
where the transpose symbol is ignored to avoid cumbersome notation.

Suppose that (without the slack variables and penalty terms)
 the corresponding dual problem has an optimal solution,
$\{X^{\ast}_i$, $Z^{\ast}$, $\dvZp^{\ast}_i-\dvZn^{\ast}_i$,
$\dvF^{\ast}_i$, $\dvG^{\ast}_i$, $\dvHp^{\ast}_i-\dvHn^{\ast}_i$,
$\forall i\}$
satisfying the Karush-Kuhn-Tucker (KKT) conditions,
\begin{align}
 &
-\nabla_{X_{i}}\Big(
f(X_{i})
    +\big\langle  \dvZp_{i}-\dvZn_{i}, X_{i}-Z\big\rangle
    +\big\langle  \dvF_{i}, \pcF_{i}(X_{i})\big\rangle
    +\big\langle  \dvG_{i},  \pcG_{i}(X_{i})\big\rangle
    +\big\langle  \dvHp_{i}-\dvHn_{i}, \pcH_{i}(X_{i})\big\rangle
\s\Big)^{\ast}
\in
\partial_{X_{i}}\Indicator_{[-u_{},u_{}]}(X^{\ast}_{i}),
\notag\\&
-\sum_{i}\big(\dvZp^{\ast}_{i}-\dvZn^{\ast}_{i}\big)
\in 
\partial_{Z}\Indicator_{[-u,u]}(Z^{\ast}),\ \
  X^{\ast}_{i}=Z^{\ast},\ \
  \pcF_{i}(X^{\ast}_{i})\leq 0,\ \
  \dvF^{\ast}_{i}\geq 0,\ \ 
  \big\langle \dvF^{\ast}_{i}, \pcF_{i}(X^{\ast}_{i})\big\rangle =0,\ \
  \pcG_{i}(X^{\ast}_{i})\leq 0,
  \notag\\&
  \dvG^{\ast}_{i}\geq 0,\ \ 
  \big\langle \dvG^{\ast}_{i}, \pcG_{i}(X^{\ast}_{i})\big\rangle =0,\ \
  \pcH_{i}(X^{\ast}_{i})=0\ \ \forall i.
\label{KKTConditions}
\end{align}

To make it computationally feasible,  $X_i$ and $Z$ are partitioned
into $M$ disjoint subvectors,
\begin{align}
X_{i}=(X_{i,1},\ldots, X_{i,M})\ \forall i=1,\ldots,N,\ \
Z=(Z_{1},\ldots, Z_{M}),
\label{ACPGSXiZiPartitioned}
\end{align}
where the dimensions and  component orders of subvectors
$\{X_{i,l}, Z_{l}\}$ are independent of $i$:
$Z_{l},X_{i,l}\in\mathbb{R}^{m_l}$, $\sum_{l=1}^{M}m_l=n$.
It is supposed that $m_{l}$ for all $l$ have similar values.
To help solve $X_{i,l}$ through 
the block coordinate Guess-Seidel method,
 we introduce 
\begin{align}
&
X^{j,+}_{i,l}
:=(X^{j}_{i,1},\ldots, X^{j}_{i,l-1}, X_{i,l}, X^{k}_{i,l+1}, \ldots, X^{k}_{i,M}),\
X^{j,k}_{i,l}
:=(X^{j}_{i,1},\ldots, X^{j}_{i,l-1}, X^{k}_{i,l}, X^{k}_{i,l+1}, \ldots, X^{k}_{i,M}),
\notag\\&
X^{j,k+1}_{i,l}
:=(X^{j}_{i,1},\ldots, X^{j}_{i,l-1},
        X^{k+1}_{i,l}, X^{k}_{i,l+1}, \ldots, X^{k}_{i,M}),\ 
X^{k}_{i}:=(X^{k}_{i,1},\ldots, \ldots, X^{k}_{i,M}),\
X^{k}:=\{X^{k}_{1},\ldots,X^{k}_{N}\},
\label{ACPGSConstraintsithCBPartitioned}
\end{align}
where the superscripts $j,k$ denote the $j/k$-th iteration
and the subscript $(i,l)$ denotes the subblock
or the process to be updated.
Instead of $X^{j,+}_{i,l}$ above,
a cumbersome notation without ambiguity like $X^{j,+,k}_{i,l}$ could be adopted too.
In the definitions of \eqref{ACPGSConstraintsithCBPartitioned},
$X_i$ can be substituted by $Z$, $\svZp_{i}$, and $\svZn_{i}$,
respectively. 

Next, $L_i$ of \eqref{ACPGSLagranginFunctionithBC} is partitioned functionally 
according to
\begin{align}
&
 L_i
=\sum_{l=1}^{M}L_{i,l}\big(X^{k+1,+}_{i,l},Z_{l},\svZp_{i,l},\svZn_{i,l},
    \svF_i,\svG_i,\svHp_i,\svHn_i,
    \dvZp_{i,l},\dvZn_{i,l},\dvF_i,\dvG_i,\dvHp_i,\dvHn_i,\rho_i\big),
\notag\\&
 L_{i,l}\big(X^{k+1,+}_{i,l},Z_{l},\svZp_{i,l},\svZn_{i,l},
    \svF_i,\svG_i,\svHp_i,\svHn_i,
    \dvZp_{i,l},\dvZn_{i,l},\dvF_i,\dvG_i,\dvHp_i,\dvHn_i,\rho_i\big)
\notag\\
=\,&
f(X^{k+1,+}_{i,l})
    +\big\langle  \dvZp_{i,l}, X_{i,l}-Z_{l}+\svZp_{i,l}\big\rangle
    +\big\langle  \dvZn_{i,l}, Z_{l}-X_{i,l}+\svZn_{i,l}\big\rangle
    +\big\langle  \dvF_{i},  \pcF_{i}(X^{k+1,+}_{i,l})+\svF_{i}\big\rangle
    \notag\\&
    +\big\langle  \dvG_{i},  \pcG_{i}(X^{k+1,+}_{i,l})+\svG_i\big\rangle
    +\big\langle  \dvHp_{i}, \pcH_{i}(X^{k+1,+}_{i,l})+\svHp_{i}\big\rangle
    +\big\langle  \dvHn_{i},-\pcH_{i}(X^{k+1,+}_{i,l})+\svHn_{i}\big\rangle
    \notag\\&
    +\frac{\rho_i}{2}\Big(
            \big\Vert X_{i,l}-Z_{l}+\svZp_{i,l} \big\Vert^2
            +\big\Vert Z_{l}-X_{i,l}+\svZn_{i,l} \big\Vert^2
            +\big\Vert \pcF_{i}(X^{k+1,+}_{i,l})+\svF_{i} \big\Vert^2
            +\big\Vert \pcG_{i}(X^{k+1,+}_{i,l})+\svG_{i} \big\Vert^2
            \notag\\&\hskip12mm
            +\big\Vert \pcH_{i}(X^{k+1,+}_{i,l})+\svHp_{i} \big\Vert^2
            +\big\Vert -\pcH_{i}(X^{k+1,+}_{i,l})+\svHn_{i} \big\Vert^2
    \Big).
\label{ACPGSLagranginFunctionithBClthSubblock}
\end{align}
Here, 
$\dvF_{i}$, $\dvG_{i}$, $\dvHp_{i}$, and $\dvHn_{i}$ are the dual variables 
in the $i$-CB.
$\dvF_{i}\geq 0$ is taken to make
$L_{i,l}$ in \eqref{ACPGSLagranginFunctionithBClthSubblock}
a convex function of $X^{}_{i,l}$.

\section{Update Rules for Primal Variables}
\label{sec:UpdatePrimal}
We update the primal variable $X_{i}$
 by applying the block coordinate Gauss-Seidel method,
the proximal point method, and ADMM,
on the basis of \eqref{ACPGSLagranginFunctionithBClthSubblock}.
Firstly,
under $i\in\{1,\ldots,N\}$ and $l\in \{1,\ldots,M\}$ fixed, at iteration $k$,
 $X_{i,l}$  of the $(i,l)$-subblock is updated through
\begin{align} 
X^{k+1}_{i,l}=
\argmin_{-u_{l}\leq X_{i,l}\leq u_{l}}\s
    \Big\{&
         L_{i,l}\big(X^{k+1,+}_{i,l},Z^{k}_{l},\svZp^{k}_{i,l},\svZn^{k}_{i,l},
                    \svF^{k}_i,\svG^{k}_i,\svHp^{k}_i,\svHn^{k}_i,
                    \dvZp^{k}_{i,l},\dvZn^{k}_{i,l},
                    \dvF^{k}_i,\dvG^{k}_i,\dvHp^{k}_i,\dvHn^{k}_i,\rho_i\big)
    \notag\\[-6pt]&
        +\proxXNormOcoef^{k}_{i,l}\big\Vert X^{}_{i,l}-X^{k}_{i,l} \big\Vert_{1}
        +\frac{\proxXNormTcoef^{k}_{i,l}}{2}\big\Vert X^{}_{i,l}-X^{k}_{i,l} \big\Vert^{2}
    \Big\},
\label{ACPGSProximalXilA}
\end{align}
where the double proximal terms are employed: 
$\proxXNormOcoef^{k}_{i,l}$ is a non-negative control parameter
associated with the 1-norm and
$\proxXNormTcoef^{k}_{i,l}$ is a positive control parameter associated with the 2-norm.
The reasons for the double proximal forms are to become clear. The values of the proximal control parameters
$\proxXNormOcoef^{k}_{i,l}$
and $\proxXNormTcoef^{k}_{i,l}$ are to be estimated.

Secondly, based on \eqref{ACPGSLagranginFunctionithBClthSubblock},
$Z$ is updated through
\begin{align}
Z^{k+1}_{l}=
\argmin_{-u_{l}\leq Z_{l}\leq u_{l}} 
        \sum_{i=1}^{N}\Big(&
             \big\langle  \dvZp^{k}_{i,l}, X^{k+1}_{i,l}-Z_{l}+\svZp^{k}_{i,l} \big\rangle
            +\big\langle  \dvZn^{k}_{i,l}, Z_{l}-X^{k+1}_{i,l}+\svZn^{k}_{i,l} \big\rangle
            \notag\\[-6pt]&
            +\frac{\rho_{i}}{2}\big(
                                     \big\Vert X^{k+1}_{i,l}-Z_{l}+\svZp^{k}_{i,l} \big\Vert^2
                                    +\big\Vert Z_{l}-X^{k+1}_{i,l}+\svZn^{k}_{i,l} \big\Vert^2
                                    \big)
            +\frac{\tau^{k}}{2}\Vert Z_{l}-Z^{k}_{l} \Vert^2\Big),
\label{ACPGSProximalZl}
\end{align}
where $\tau^{k}$ is a positive proximal control parameter whose value is to be fixed.

Considering that the summation operation in \eqref{ACPGSProximalZl}
may be too big to be implemented in a single process,
 the following simple rudimentary procedure may be designed to realize it:
($\big(\sum_{j=1}^{0} (\tau^{k}+2\rho_{j})\big) Z_{0,l}$ is set to zero.)
\begin{align}
&
\text{In the $(i,l)$-process, $i=1,\ldots,N-1$},
\notag\\&
\hskip7mm \text{$\Big(\sum_{j=1}^{i} (\tau^{k}+2\rho_{j})\Big) Z_{i,l}:=
\Big(\sum_{j=1}^{i-1} (\tau^{k}+2\rho_{j})\Big) Z_{i-1,l}
        +2\rho_{i}X^{k+1}_{i,l}
             +\rho_{i}(\svZp^{k}_{i,l}-\svZn^{k}_{i,l})+\dvZp^{k}_{i,l}-\dvZn^{k}_{i,l}
                        +\tau^{k}Z^{k}_{l}$,}
             \notag\\&\hskip47mm 
             \text{and pass to the $(i+1,l)$-process.}
\notag\\&
\text{In the $(N,l)$-process,}
\notag\\&\hskip7mm
\text{$\Big(\sum_{j=1}^{N} (\tau^{k}+2\rho_{j})\Big) Z_{N,l}:=
\Big(\sum_{j=1}^{N-1} (\tau^{k}+2\rho_{j})\Big) Z_{N-1,l}
        +2\rho_{N}X^{k+1}_{N,l}
             +\rho_{N}(\svZp^{k}_{N,l}-\svZn^{k}_{N,l})+\dvZp^{k}_{N,l}-\dvZn^{k}_{N,l}
                        +\tau^{k}Z^{k}_{l}$,} 
 \notag\\& \hskip7mm
\text{$Z^{k+1}_{l}:=P_{[-u_{l},u_{l}]}(Z_{N,l})$ and pass to the $(i,l)$-processes for all
                    $i=1,\ldots,N-1$.}
\label{ACPGSProximalZkAlgorithm}
\end{align}
Here, the symbols $Z_{i,l}$ are used as the intermediate variables
to help presentation, with $i$ indicating the computation done in the $i$-CB,
$P_{[-u_{l},u_{l}]}$ is the projection into $[-u_{l},u_{l}]$.

Since all the slack variables hold for the whole $i$-CB,
 $X^{k+1}_{i}$ and $Z^{k+1}$ are used to update them.
Fix $i\in\{1,\ldots,N\}$.
\begin{align} 
\svZp^{k+1}_{i}=
\argmin_{0\leq \svZp_{i}\leq \usvZp_{Y_{i}}}\s
\Big\{
 \big\langle \dvZp^{k}_{i}, X^{k+1}_{i}-Z^{k+1}_{}+\svZp_{i} \big\rangle
+\frac{\rho_{i}}{2}\big\Vert X^{k+1}_{i}-Z^{k+1}_{}+\svZp_{i} \big\Vert^2
+\frac{\svZpcoef^{k}_{i}}{2}\big\Vert \svZp_{i}-\svZp^{k}_{i} \big\Vert^2
    \Big\},
\label{ACPGSProximalsvZpiA}
\end{align}
\vspace*{-4mm}
\begin{align} 
\svZn^{k+1}_{i}=
\argmin_{0\leq \svZn_{i}\leq \usvZn_{Y_{i}}}\s
\Big\{
 \big\langle \dvZn^{k}_{i}, Z^{k+1}_{}-X^{k+1}_{i}+\svZn_{i} \big\rangle
+\frac{\rho_{i}}{2}\big\Vert Z^{k+1}_{}-X^{k+1}_{i}+\svZn_{i} \big\Vert^2
+\frac{\svZncoef^{k}_{i}}{2}\big\Vert \svZn_{i}-\svZn^{k}_{i} \big\Vert^2
    \Big\},
\label{ACPGSProximalsvZniA}
\end{align}
\vspace*{-4mm}
\begin{align}
\svF^{k+1}_{i}=
\argmin_{0\leq \svF_{i}\leq \usvF_{Y_{i}}}\s
\Big\{\,&
 \big\langle \dvF^{k}_{i}, \pcF_{i}(X^{k+1}_{i})+\svF_{i} \big\rangle
+\frac{\rho_{i}}{2}\big\Vert \pcF_{i}(X^{k+1}_{i})+\svF^{}_{i} \big\Vert^2
+\frac{\svFcoef^{k}_{i}}{2}\big\Vert \svF_{i}-\svF^{k}_{i} \big\Vert^2
    \Big\},
\label{ACPGSProximalsvFiA}
\end{align}
\vspace*{-4mm}
\begin{align}
\svG^{k+1}_{i}=
\argmin_{0\leq \svG_{i}\leq \usvG_{Y_{i}}}\s
\Big\{\,&
 \big\langle \dvG^{k}_{i}, \pcG_{i}(X^{k+1}_{i})+\svG_{i} \big\rangle
+\frac{\rho_{i}}{2} \big\Vert \pcG_{i}(X^{k+1}_{i})+\svG_{i} \big\Vert^2
+\frac{\svGcoef^{k}_{i}}{2}\big\Vert \svG_{i}-\svG^{k}_{i} \big\Vert^2
    \Big\},
\label{ACPGSProximalsvGiA}
\end{align}
\vspace*{-4mm}
\begin{align}
\svHp^{k+1}_{i}=
\argmin_{0\leq \svHp_{i}\leq \usvHp_{Y_{i}}}\s
\Big\{
 \big\langle \dvHp^{k}_{i}, \pcH_{i}(X^{k+1}_{i})+\svHp_{i} \big\rangle
+\frac{\rho_{i}}{2} \big\Vert \pcH_{i}(X^{k+1}_{i})+\svHp_{i} \big\Vert^2
+\frac{\svHpcoef^{k}_{i}}{2}\big\Vert \svHp_{i}-\svHp^{k}_{i} \big\Vert^2
    \Big\},
\label{ACPGSProximalsvHpiA}
\end{align}
and
\begin{align} 
\svHn^{k+1}_{i}=
\argmin_{0\leq \svHn_{i}\leq \usvHn_{Y_{i}}}\s
\Big\{
 \big\langle \dvHn^{k}_{i}, -\pcH_{i}(X^{k+1}_{i})+\svHn_{i} \big\rangle
+\frac{\rho_{i}}{2} \big\Vert -\pcH_{i}(X^{k+1}_{i})+\svHn_{i} \big\Vert^2
+\frac{\svHncoef^{k}_{i}}{2} \big\Vert \svHn_{i}-\svHn^{k}_{i} \big\Vert^2
    \Big\}.
\label{ACPGSProximalsvHniA}
\end{align}
Here, $\svZpcoef^{k}_{i}$,  $\svZncoef^{k}_{i}$,
   $\svFcoef^{k}_{i}$, $\svGcoef^{k}_{i}$,
 $\svHpcoef^{k}_{i}$, and $\svHncoef^{k}_{i}$
 are the positive proximal control parameters to be fixed.

\section{Dual Updates and Convergence Analysis}
\label{sec:UpdateDual}

To analyze the convergence of the algorithm composed of 
\eqref{ACPGSProximalXilA}, \eqref{ACPGSProximalZl},
\eqref{ACPGSProximalsvZpiA} through \eqref{ACPGSProximalsvHniA}
and introduce the dual updates,
we apply the first-order characterization of convex functions
to the functions involved in the primal updates above.

First, application of Fermat's rule to \eqref{ACPGSProximalXilA},
 \eqref{ACPGSProximalZl},
\eqref{ACPGSProximalsvZpiA} through \eqref{ACPGSProximalsvHniA} gives
\begin{align}
&
-\nabla_{X_{i,l}}\Big\{
     \big\langle  \dvF^{k}_{i},\pcF_{i}(X^{k+1,+}_{i,l}) \big\rangle
  +\frac{\rho_i}{2}\Big(
     \big\Vert X_{i,l}-Z^{k}_{l}+\svZp^{k}_{i,l} \big\Vert^2
    +\big\Vert Z^{k}_{l}-X_{i,l}+\svZn^{k}_{i,l} \big\Vert^2
    +\big\Vert \pcF_{i}(X^{k+1,+}_{i,l})+\svF^{k}_{i} \big\Vert^2
    \notag\\[-0pt]&\hskip53mm
    +\big\Vert \pcG_{i}(X^{k+1,+}_{i,l})+\svG^{k}_{i} \big\Vert^2
    +\big\Vert \pcH_{i}(X^{k+1,+}_{i,l})+\svHp^{k}_{i} \big\Vert^2
    +\big\Vert -\pcH_{i}(X^{k+1,+}_{i,l})+\svHn^{k}_{i} \big\Vert^2
    \Big)
    \notag\\[-0pt]&\hskip15mm
+\frac{\proxXNormTcoef^{k}_{i,l}}{2}\big\Vert X^{}_{i,l}-X^{k}_{i,l} \big\Vert^{2}
\Big\}_{X^{k+1}_{i,l}}
\in
     (f_{l})^{T}
    +\dvZp^{k}_{i,l}-\dvZn^{k}_{i,l}
    +(\pcG_{i,l})^{T}\dvG^{k}_{i}
    +(\pcH_{i,l})^{T} \big(\dvHp^{k}_{i}-\dvHn^{k}_{i}\big)
\notag\\[-0pt]&\hskip60mm
 +\partial_{X_{i,l}}
    \Big\{
     \proxXNormOcoef^{k}_{i,l}\big\Vert X^{}_{i,l}-X^{k}_{i,l} \big\Vert_{1}
    +\Indicator_{[-u_{l}, u_{l}]}(X_{i,l})
\Big\}_{X^{k+1}_{i,l}},
\label{SolvingXil}
\end{align}
\vspace*{-4mm}
\begin{align}
-\frac{1}{N}\sum_{i=1}^{N}\Big(
         \dvZn^{k}_{i,l}-\dvZp^{k}_{i,l}
        +\rho_{i}\big(
                      2Z^{k+1}_{l}-2X^{k+1}_{i,l}-\svZp^{k}_{i,l}+\svZn^{k}_{i,l}
                                    \big)
        +\tau^{k}(Z^{k+1}_{l}-Z^{k}_{l})
        \Big)
\in \partial_{Z_{l}}\Indicator_{[-u_{l},u_{l}]}(Z^{k+1}_{l}),
\label{SolvingZk}
\end{align}
\vspace*{-4mm}
\begin{align}
-\Big(
 \dvZp^{k}_{i}
+\rho_{i}\big(X^{k+1}_{i}-Z^{k+1}_{}+\svZp^{k+1}_{i}\big)
+\svZpcoef^{k}_{i}\big(\svZp^{k+1}_{i}-\svZp^{k}_{i}\big)
\Big)
\in
\partial_{\svZp_{i}}\Indicator_{[0,\usvZp_{Y_{i}}]}(\svZp^{k+1}_{i}),
\label{SolvingsvZpk}
\end{align}
\vspace*{-4mm}
\begin{align}
-\Big( \dvZn^{k}_{i}
    +\rho_{i}\big( Z^{k+1}_{}-X^{k+1}_{i}+\svZn^{k+1}_{i} \big)
    +\svZncoef^{k}_{i}( \svZn^{k+1}_{i}-\svZn^{k}_{i} )
\Big)
\in
\partial_{\svZn_{i}}\Indicator_{[0,\usvZn_{Y_{i}}]}(\svZn^{k+1}_{i}),
\label{SolvingsvZnk}
\end{align}
\vspace*{-4mm}
\begin{align}
-\Big( \dvF^{k}_{i}
+\rho_{i}\big( \pcF_{i}(X^{k+1}_{i})+\svF^{k+1}_{i} \big)
+\svFcoef^{k}_{i}( \svF^{k+1}_{i}-\svF^{k}_{i} )
\Big)
\in
\partial_{\svF_{i}}\Indicator_{[0,\usvF_{Y_{i}}]}(\svF^{k+1}_{i}),
\label{SolvingsvFk}
\end{align}
\vspace*{-4mm}
\begin{align}
-\Big(
 \dvG^{k}_{i}
+\rho_{i}\big( \pcG_{i}(X^{k+1}_{i})+\svG^{k+1}_{i} \big)
+\svGcoef^{k}_{i}( \svG^{k+1}_{i}-\svG^{k}_{i} )
\Big)
\in
\partial_{\svG_{i}}\Indicator_{[0,\usvG_{Y_{i}}]}(\svG^{k+1}_{i}),
\label{SolvingsvGk}
\end{align}
\vspace*{-4mm}
\begin{align}
-\Big(
 \dvHp^{k}_{i}
+\rho_{i}\big( \pcH_{i}(X^{k+1}_{i})+\svHp^{k+1}_{i} \big)
+\svHpcoef^{k}_{i}( \svHp^{k+1}_{i}-\svHp^{k}_{i} )
\Big)
\in
\partial_{\svHp_{i}}\Indicator_{[0,\usvHp_{Y_{i}}]}(\svHp^{k+1}_{i}),
\label{SolvingsvHpk}
\end{align}
and
\begin{align} 
-\Big(
 \dvHn^{k}_{i}
+\rho_{i}\big( -\pcH_{i}(X^{k+1}_{i})+\svHn^{k+1}_{i} \big)
+\svHncoef^{k}_{i}( \svHn^{k+1}_{i}-\svHn^{k}_{i} )
\Big)
\in
\partial_{\svHn_{i}}\Indicator_{[0,\usvHn_{Y_{i}}]}(\svHn^{k+1}_{i}).
\label{SolvingsvHnk}
\end{align}
Here, 
$\Indicator_{[-u_{l}, u_{l}]}(X_{i,l})$ and so on are the indicator functions.
 $f(Z):=\s\sum_{l}f_{l}Z_{l}$,
 $\pcG_{i}(Z):=\s\sum_{l}\pcG_{i,l}Z_{l}+\pcG_{i,0}$, and
$\pcH_{i}(Z):=\s\sum_{l}\pcH_{i,l}Z_{l}+\pcH_{i,0}$ are used,
and the components of $\pcF_{i}(Z)$ are
specified by \eqref{NLGSQuadraticConvexConstraints}.

Equations~\eqref{SolvingXil} through \eqref{SolvingsvHnk}
indicate the need to use
 $\{f(Z),(\pcF(Z))_{j},(\pcG(Z))_{j},(\pcH(Z))_{j}: Z\in[-u,u]\}$
 and their derivatives.
(Considering that the quantities in the set listed
have different number of components, the same index $j$ is used
in a symbolic manner for brevity.)
In the present formulation, as part of the preconditioning,
all the quantities are supposedly scalable computationally to get the following.
\begin{assumption}
$\qquad$
\begin{enumerate}
 \item[(A)]
$u=[1,1,\ldots,1]^T$.
$f(Z)=\sum_{l}f_{l}Z_{l}$,
$\max_{l}\Vert f_{l} \Vert_{\infty}=1$.
\item[(B)]
Let $g(Z)=\sum_{l}g_{l}Z_{l}+g_0$
be a component of $\{\pcG_{i}(Z),\pcH_{i}(Z)\}$.
$\max_{Z\in[-u,u]}\s|g(Z)|=1$,
 $\sum_{l}\s\Vert g_{l}\Vert_{1}+\vert g_0 \vert = 1$.
\item[(C)]
Let $e(Z)$ be a component of $\pcF_{i}(Z)$.
$\max_{z\in[-u,u]}|e(Z)|\leq 1$.
\item[(D)]
Owing to the presence of the gradient-related terms,
$\{\langle\dvF^{k}_{i},\nabla_{X_{i,l}}\pcF_{i}(X^{k+1,k+1}_{i,l})\rangle$,\,$(\pcG_{i,l})^{T}\dvG^{k}_{i}$,\,$(\pcH_{i,l})^{T}(\dvHp^{k}_{i}-\dvHn^{k}_{i})$,
$\langle \pcF_{i}(X^{k+1,k+1}_{i,l})+\svF^{k}_{i},
                  \nabla_{X_{i,l}}\pcF_{i}(X^{k+1,k+1}_{i,l})\rangle$,
 $(\pcG_{i,l})^{T}(\pcG_{i}(X^{k+1,k+1}_{i,l})+\svG^{k}_{i})$,
  $(\pcH_{i,l})^{T}(2\pcH_{i}(X^{k+1,k+1}_{i,l})
                          +\svHp^{k}_{i}-\svHn^{k}_{i})\}$
  involving summations of all the components
in $\{\pcF_{i}(Z)$,\,$\pcG_{i}(Z)$,\,$\pcH_{i}(Z)\}$,
there is a necessity to restrict
the sizes of $\{\pcF_{i}(Z)$,\,$\pcG_{i}(Z)$,\,$\pcH_{i}(Z)\}$
selected from $\{\pcF(Z)$,\,$\pcG(Z)$,\,$\pcH(Z)\}$
in order to control the magnitudes of these terms
and $\proxXNormTcoef^{k}_{i,l}$ in \eqref{SolvingXil}.
\end{enumerate}
\label{AssumptionSetI}
\end{assumption}
Here, all the upper bounds are tentatively taken in order to analyze
the algorithm proposed and can be scaled up, if needed computationally.
The precondition of $f(Z)$ above intends to
make the $(f_{l})^T$-term in
the $X^{k+1}_{i,l}$-update~\eqref{SolvingXil} not too low
relative to the rest from a computational perspective,
because $(f_{l})^T$ has to balance the rest terms
in \eqref{SolvingXil} that involve the summations mentioned in
Assumption~\ref{AssumptionSetI}(D).

Next, the first-order characteristics of the convex functions involved in
\eqref{SolvingXil} through \eqref{SolvingsvHnk} can be constructed
straight-forwardly,
under $X_{i,l}=X^{k}_{i,l}$, $Z_{l}=Z^{k}_{l}$, $\svZp_{i}=\svZp^{k}_{i}$,
 $\svZn_{i}=\svZn^{k}_{i}$, 
 $\svF_{i}=\svF^{k}_{i}$, $\svG_{i}=\svG^{k}_{i}$,
 $\svHp_{i}=\svHp^{k}_{i}$, and $\svHn_{i}=\svHn^{k}_{i}$.
 With the help of $X^{k+1,k}_{i,l}=X^{k+1,k+1}_{i,l-1}$,
 summation of the characteristics results in the following.
\begin{lemma} 
The primal updates~\eqref{ACPGSProximalXilA}, \eqref{ACPGSProximalZl},
\eqref{ACPGSProximalsvZpiA} through \eqref{ACPGSProximalsvHniA} yield, for all $k$,
\begin{align}
L^{k}
\geq
L^{k+1}
+D^{k}
+P^{k}
+\sum_{i}\sum_{l}
    \Big( \proxXNormOcoef^{k}_{i,l}\,\big\Vert X^{k}_{i,l}-X^{k+1}_{i,l}\big\Vert_{1}
         +U^{k}_{i,l}\Big).
\label{GlobalInequality}
\end{align}
Here, $L^{k}$ is the augmented Lagrangian function sequence defined through
\begin{align*}
L^{k}_{i}:=\,&
     f(X^{k}_{i})
    +\big\langle \dvZp^{k}_{i}, X^{k}_{i}-Z^{k}_{}+\svZp^{k}_{i} \big\rangle
    +\big\langle \dvZn^{k}_{i}, Z^{k}_{}-X^{k}_{i}+\svZn^{k}_{i} \big\rangle
    +\big\langle \dvF^{k}_{i}, \pcF_{i}(X^{k}_{i})+\svF^{k}_{i} \big\rangle
    \notag\\[3pt]&
    +\big\langle \dvG^{k}_{i}, \pcG_{i}(X^{k}_{i})+\svG^{k}_{i} \big\rangle
    +\big\langle \dvHp^{k}_{i}, \pcH_{i}(X^{k}_{i})+\svHp^{k}_{i}\big\rangle
    +\big\langle \dvHn^{k}_{i}, -\pcH_{i}(X^{k}_{i})+\svHn^{k}_{i} \big\rangle
    \notag\\[3pt]&
    +\frac{\rho_{i}}{2}\Big(\Vert X^{k}_{i}-Z^{k}_{}+\svZp^{k}_{i} \Vert^2 
            +\Vert Z^{k}_{}-X^{k}_{i}+\svZn^{k}_{i} \Vert^2 
            +\Vert \pcF_{i}(X^{k}_{i})+\svF^{k}_{i} \Vert^2 
            +\Vert \pcG_{i}(X^{k}_{i})+\svG^{k}_{i} \Vert^2 
            \notag\\&\hskip11mm
            +\Vert \pcH_{i}(X^{k}_{i})+\svHp^{k}_{i} \Vert^2 
            +\Vert -\pcH_{i}(X^{k}_{i})+\svHn^{k}_{i} \Vert^2 
    \Big),\ \ 
 L^{k}:=\sum_{i}L^{k}_{i};
\end{align*}
the quantity $D^{k}$ is defined through
\begin{align*}
D^{k}_{i}
:=\,&
 \big\langle \dvZp^{k}_{i}-\dvZp^{k+1}_{i},  X^{k+1}_{i}-Z^{k+1}_{}+\svZp^{k+1}_{i} \big\rangle
+\big\langle \dvZn^{k}_{i}-\dvZn^{k+1}_{i}, Z^{k+1}_{}-X^{k+1}_{i}+\svZn^{k+1}_{i} \big\rangle
\notag\\[3pt]&
+\big\langle \dvF^{k}_{i}-\dvF^{k+1}_{i}, \pcF_{i}(X^{k+1}_{i})+\svF^{k+1}_{i} \big\rangle
+\big\langle \dvG^{k}_{i}-\dvG^{k+1}_{i}, \pcG_{i}(X^{k+1}_{i})+\svG^{k+1}_{i} \big\rangle
\notag\\[3pt]&
+\big\langle \dvHp^{k}_{i}-\dvHp^{k+1}_{i}, \pcH_{i}(X^{k+1}_{i})+\svHp^{k+1}_{i}\big\rangle
+\big\langle \dvHn^{k}_{i}-\dvHn^{k+1}_{i},-\pcH_{i}(X^{k+1}_{i})+\svHn^{k+1}_{i} \big\rangle,\ \ 
D^{k}:=\sum_{i}D^{k}_{i},
\end{align*}
which helps to develop the update rules for the duals,
$\{\dvZp^{k+1}_{i}, \dvZn^{k+1}_{i}, \dvF^{k+1}_{i}, 
\dvG^{k+1}_{i}, \dvHp^{k+1}_{i}, \dvHn^{k+1}_{i}\}$;
the quantity $P^{k}$ contains all the quadratic terms nonnegative,
\begin{align*}
 P^{k}_{i}
 :=\,&
\sum_{l}\proxXNormTcoef^{k}_{i,l}\big\Vert X^{k+1}_{i,l}-X^{k}_{i,l} \big\Vert^2
+\rho_i\big\Vert X^{k+1}_{i}-X^{k}_{i} \big\Vert^2
+(\tau^{k}+\rho_{i})\big\Vert Z^{k+1}_{}-Z^{k}_{}\big\Vert^2
+\big(\svZpcoef^{k}_{i}+\frac{\rho_{i}}{2}\big)\big\Vert \svZp^{k+1}_{i}-\svZp^{k}_{i}\big\Vert^2
\notag\\[-4pt]&
+\big(\svZncoef^{k}_{i}+\frac{\rho_{i}}{2}\big)\big\Vert\svZn^{k+1}_{i}-\svZn^{k}_{i}\Vert^2
+\big(\svFcoef^{k}_{i}+\frac{\rho_{i}}{2}\big)
                \big\Vert \svF^{k+1}_{i}-\svF^{k}_{i} \big\Vert^2
+\big(\svGcoef^{k}_{i}+\frac{\rho_{i}}{2}\big)\big\Vert \svG^{k+1}_{i}-\svG^{k}_{i} \big\Vert^2
\notag\\&
+\big(\svHpcoef^{k}_{i}+\frac{\rho_{i}}{2}\big)\big\Vert \svHp^{k+1}_{i}-\svHp^{k}_{i}\big\Vert^2
+\big(\svHncoef^{k}_{i}+\frac{\rho_{i}}{2}\big)\big\Vert \svHn^{k+1}_{i}-\svHn^{k}_{i} \big\Vert^2
\notag\\&
+\frac{\rho_{i}}{2}\sum_{l}
  \Big(              \big\Vert \pcF_{i}(X^{k+1,k+1}_{i,l})-\pcF_{i}(X^{k+1,k}_{i,l}) \big\Vert^2
                    +\big\Vert\pcG_{i,l}(X^{k+1}_{i,l}-X^{k}_{i,l}) \big\Vert^2
                    +2\big\Vert \pcH_{i,l}(X^{k+1}_{i,l}-X^{k}_{i,l}) \big\Vert^2
    \Big),\ \ P^{k}:=\sum_{i} P^{k}_{i};
\end{align*}
the quantity $U^{k}_{i,l}$ is defined through
\begin{align}
&
 U^{k}_{i,l}=U_{i,l}(\dvF^{k}_{i},\svF^{k}_{i},
                X^{k+1,k+1}_{i,l},X^{k+1,k}_{i,l})
 \notag\\[3pt]
:=\,&
 \big\langle \dvF^{k}_{i}+\rho_i\big(\pcF_{i}(X^{k+1,k+1}_{i,l})+\svF^{k}_{i}\big),
                \pcF_{i}(X^{k+1,k}_{i,l})-\pcF_{i}(X^{k+1,k+1}_{i,l})
                -(X^{k}_{i,l}-X^{k+1}_{i,l})^{T}\,\nabla_{X_{i,l}}\pcF_{i}(X^{k+1,k+1}_{i,l})
            \big\rangle,
\label{UkilDefn}
\end{align}
resulting from the nonlinearity of the constraints of inequality
 $\pcF_{i}(Z)\leq 0$.
\label{GlobalInequalitykLemma}
\end{lemma}

Motivated by the structure of $D^{k}_{i}$,
the following descent models are proposed to update the dual sequences,
\begin{align}
&
  \dvZp^{k+1}_{i}:=\dvZp^{k}_{i}
  -\dvZpcoef^{k+1}_{i}\big(X^{k+1}_{i}-Z^{k+1}_{}+\svZp^{k+1}_{i}\big),\ \
  \dvZn^{k+1}_{i}:=\dvZn^{k}_{i}
  -\dvZncoef^{k+1}_{i}\big(Z^{k+1}_{}-X^{k+1}_{i}+\svZn^{k+1}_{i}\big),
\notag\\[3pt]&
   \dvF^{k+1}_{i}:=\dvF^{k}_{i}
  -\dvFcoef^{k+1}_{i}\big(\pcF_{i}(X^{k+1}_{i})+\svF^{k+1}_{i}\big),\ \
  \dvG^{k+1}_{i}:=\dvG^{k}_{i}
  -\dvGcoef^{k+1}_{i}\big(\pcG_{i}(X^{k+1}_{i})+\svG^{k+1}_{i}\big),
\notag\\[3pt]&
  \dvHp^{k+1}_{i}:=\dvHp^{k}_{i}
  -\dvHpcoef^{k+1}_{i}\big(\pcH_{i}(X^{k+1}_{i})+\svHp^{k+1}_{i}\big),\ \
  \dvHn^{k+1}_{i}:=\dvHn^{k}_{i}
  -\dvHncoef^{k+1}_{i}\big(-\pcH_{i}(X^{k+1}_{i})+\svHn^{k+1}_{i}\big).
\label{DualUpdateRules}
\end{align}
The dual coefficients, 
$\dvZpcoef^{k+1}_{i}$, $\dvZncoef^{k+1}_{i}$,
 $\dvFcoef^{k+1}_{i}$, $\dvGcoef^{k+1}_{i}$,
$\dvHpcoef^{k+1}_{i}$, and $\dvHncoef^{k+1}_{i}$
are the semi-positively definite diagonal matrices specified through
\begin{align}
&
\dvZpcoef^{k+1}_{i,j}=
\l\{\begin{array}{ll}
\dvZpcoef_{i}, & 
                     \text{if $\tdvZp^{k+1}_{i,j}:=\dvZp^{k}_{i,j}
                     -\dvZpcoef_{i} \big(X^{k+1}_{i}-Z^{k+1}_{}+\svZp^{k+1}_{i}\big)_{j}
                        \in [0, \udvZp_{\mu_{i},j}]$,}\\[3pt]
0, & \text{else};
    \end{array}\r. 
\notag\\&
\dvZncoef^{k+1}_{i,j}=
\l\{\begin{array}{ll}
\dvZncoef_{i}, & 
                     \text{if $\tdvZn^{k+1}_{i,j}:=\dvZn^{k}_{i,j}
                     -\dvZncoef_{i} \big(Z^{k+1}_{}-X^{k+1}_{i}+\svZn^{k+1}_{i}\big)_{j}
                        \in [0, \udvZn_{\mu_{i},j}]$,}\\[3pt]
0, & \text{else};
    \end{array}\r. 
\notag\\&
\dvFcoef^{k+1}_{i,j}=
\l\{\begin{array}{ll}
\dvFcoef_{i}, & 
                     \text{if $\tdvF^{k+1}_{i,j}:=\dvF^{k}_{i,j}
                     -\dvFcoef_{i} \big(\pcF_{i}(X^{k+1}_{i})+\svF^{k+1}_{i}\big)_{j}
                        \in [0, \udvF_{\mu_{i},j}]$,}\\[3pt]
0, & \text{else};
    \end{array}\r. 
\notag\\&
\dvGcoef^{k+1}_{i,j}=
\l\{\begin{array}{ll}
\dvGcoef_{i}, & 
                     \text{if $\tdvG^{k+1}_{i,j}:=\dvG^{k}_{i,j}
                     -\dvGcoef_{i} \big(\pcG_{i}(X^{k+1}_{i})+\svG^{k+1}_{i}\big)_{j}
                        \in [0, \udvG_{\mu_{i},j}]$,}\\[3pt]
0, & \text{else};
    \end{array}\r. 
\notag\\&
\dvHpcoef^{k+1}_{i,j}=
\l\{\begin{array}{ll}
\dvHpcoef_{i}, & 
                     \text{if $\tdvHp^{k+1}_{i,j}:=\dvHp^{k}_{i,j}
                     -\dvHpcoef_{i} \big(\pcH_{i}(X^{k+1}_{i})+\svHp^{k+1}_{i}\big)_{j}
                        \in [0, \udvHp_{\mu_{i},j}]$,}\\[3pt]
0, & \text{else};
    \end{array}\r. 
\notag\\&
\dvHncoef^{k+1}_{i,j}=
\l\{\begin{array}{ll}
\dvHncoef_{i}, & 
                     \text{if $\tdvHn^{k+1}_{i,j}:=\dvHn^{k}_{i,j}
                     -\dvHncoef_{i} \big(-\pcH_{i}(X^{k+1}_{i})+\svHn^{k+1}_{i}\big)_{j}
                        \in [0, \udvHn_{\mu_{i},j}]$,}\\[3pt]
0, & \text{else};
    \end{array}\r. 
\label{DualUpdateRulesdvEFGBoundsj}
\end{align}
where the dual scalar coefficients,
$\dvZpcoef_{i}$,$\dvZncoef_{i}$,$\dvFcoef_{i}$,$\dvGcoef_{i}$,$\dvHpcoef_{i}$, and $\dvHncoef_{i}$
are positive scalar constants.
These descent dual models with built-in bounds are adopted
to help make the limits of the sequences satisfy the constraints
listed in \eqref{ACPGSConstraintsithCBSlacked},
in contrast to ascent iterations \cite{Bertsekas1996};
they are made possible by the slack variables, the proximal terms,
  and the parameter and initial values
 to be addressed below.
 The treatment is like that of
 \cite{SunSun2024} which involves highly nonconvex constraints.
 Expressions~\eqref{DualUpdateRulesdvEFGBoundsj} provide a ground to introduce
 the following definitions:
A dual component sequence, say $\{\dvF^{k}_{i,j}\}_{k}$
is trapped in the neighborhood of its upper bound (NUB)
or in the neighborhood of its lower bound (NLB),
if there exists $K$ such that
$\{\dvF^{k+1}_{i,j}=\dvF^{K}_{i,j}: \tdvF^{k+1}_{i,j}>\udvF_{\mu_{i},j}\}_{k\geq K}$
or $\{\dvF^{k+1}_{i,j}=\dvF^{K}_{i,j}: \tdvF^{k+1}_{i,j}<0\}_{k\geq K}$.

Substitution of \eqref{DualUpdateRules} into Lemma~\ref{GlobalInequalitykLemma} gives
\begin{lemma}
\begin{align}
L^{k}
\geq
L^{k+1}
+\sum_{i}\sum_{j}D^{k}_{i,j}
+P^{k}
+\sum_{i}\sum_{l}
    \Big( \proxXNormOcoef^{k}_{i,l}\,\big\Vert X^{k}_{i,l}-X^{k+1}_{i,l}\big\Vert_{1}
         +U^{k}_{i,l}\Big)\ \ \forall k,
\label{GlobalInequalitydvk}
\end{align}
\vspace*{-6mm}
\begin{align}
L^{0}
\geq\,&
 L^{K}
+\sum_{k=0}^{K-1}\sum_{i}\sum_{j} D^{k}_{i,j}
+\sum_{k=0}^{K-1}P^{k}
+\sum_{k=0}^{K-1}\sum_{i}\sum_{l}
    \Big( \proxXNormOcoef^{k}_{i,l}\,\Vert X^{k}_{i,l}-X^{k+1}_{i,l}\Vert_{1}
         +U^{k}_{i,l}\Big)\ \ \forall K\in\mathbb{N},
\label{GlobalInequalitydvSum2K}
\end{align}
where
\begin{align}
D^{k}_{i,j}
:=\,&
 \dvZpcoef^{k+1}_{i,j}\big\vert \big(X^{k+1}_{i}-Z^{k+1}_{}+\svZp^{k+1}_{i}\big)_{j}\big\vert^2
+\dvZncoef^{k+1}_{i,j}\big\vert \big(Z^{k+1}_{}-X^{k+1}_{i}+\svZn^{k+1}_{i}\big)_{j}\big\vert^2
+\dvFcoef^{k+1}_{i,j}\big\vert \big(\pcF_{i}(X^{k+1}_{i})+\svF^{k+1}_{i}\big)_{j}\big\vert^2
\notag\\&\hspace{-7mm}
+\dvGcoef^{k+1}_{i,j}\big\vert \big(\pcG_{i}(X^{k+1}_{i})+\svG^{k+1}_{i}\big)_{j}\big\vert^2
+\dvHpcoef^{k+1}_{i,j}\big\vert \big(\pcH_{i}(X^{k+1}_{i})+\svHp^{k+1}_{i}\big)_{j}\big\vert^2
+\dvHncoef^{k+1}_{i,j}\big\vert \big(-\pcH_{i}(X^{k+1}_{i})+\svHn^{k+1}_{i}\big)_{j}\big\vert^2.
\label{DkijDefn}
\end{align}
\label{GlobalInequalitydvSum2KLemma}
\end{lemma}
\vspace*{-4mm}

The inequalities~\eqref{GlobalInequalitydvk}
and \eqref{GlobalInequalitydvSum2K} are viewed as the global inequalities
that provide a basis to discuss the feasibility conditions of the algorithm to be defined,
 to estimate the parameter values,
 and to initialize the primal and dual sequences.
 To this end, it is necessary to have
\vspace*{-1mm}
\begin{align}
{J}^{k}:=
 \sum_{i}\sum_{j}D^{k}_{i,j}
+P^{k}
+\sum_{i}\sum_{l}\Big(\proxXNormOcoef^{k}_{i,l}\,\Vert X^{k}_{i,l}-X^{k+1}_{i,l}\Vert_{1}
                +U^{k}_{i,l} \Big)> 0
                \ \ \forall k,
\label{UkBounded}
\end{align}
such that the sequence
$\{L^{k}\}_{k}$ is bounded from above by $L^{0}$,
monotonically decreasing,
and bounded from below in a certain manner to be specified.
For the sake of convenience, we introduce the extended sequences,
 \begin{align}
  &
  \eZp^{k}_{i}:=X^{k}_{i}-Z^{k}_{}+\svZp^{k}_{i},\ \
  \eZn^{k}_{i}:=Z^{k}_{}-X^{k}_{i}+\svZn^{k}_{i},\ \
  \eF^{k}_{i}:=\pcF_{i}(X^{k}_{i})+\svF^{k}_{i},\ \
  \eG^{k}_{i}:=\pcG_{i}(X^{k}_{i})+\svG^{k}_{i},
  \notag\\&
  \eHp^{k}_{i}:=\pcH_{i}(X^{k}_{i})+\svHp^{k}_{i},\ \
  \eHn^{k}_{i}:=-\pcH_{i}(X^{k}_{i})+\svHn^{k}_{i}.
  \label{ExtendedSequecnesDfn}
 \end{align}

Firstly, to have \eqref{UkBounded} hold with the data available at the $k$-th iteration,
 the proximal parameter $\proxXNormOcoef^{k}_{i,l}$
is specified through an one-step-delayed response,
\begin{align}
 \proxXNormOcoef^{k}_{i,l}
 =
\l\{\begin{array}{ll}
      0,                      & \text{if $U^{k-1}_{i,l} \geq 0$},\\[4pt]
\Gamma^{k}_{i,l}\,\vert U^{k-1}_{i,l} \vert/\Vert X^{k}_{i,l}-X^{k-1}_{i,l}\Vert_{1},  & \text{else}.
     \end{array}
\r.
\label{proxXNormOcoefkilDefn}
\end{align}
$\Gamma^{k}_{i,l}\s\geq\s 1$ is constant and preferred to be moderate
to help hold \eqref{GlobalInequalitydvk};
$\proxXNormOcoef^{0}_{i,l}= 0$.
Function~\eqref{proxXNormOcoefkilDefn} is well-defined, according to
\begin{align}
 \big\vert U^{k}_{i,l}\big\vert
 \leq
\GammaF_{i}
       \Big( \big\Vert \dvF^{k}_{i}\big\Vert_{\infty}
            +\rho_{i}\big\Vert \pcF_{i}(X^{k+1,k+1}_{i,l})+\svF^{k}_{i} \big\Vert_{\infty}\Big)
        \big\Vert X^{k+1}_{i,l}-X^{k}_{i,l} \big\Vert_{1},
\label{GlobalInequalityUkilBounded}
\end{align}
where the constant $\GammaF_{i}$ depends on
 the structure of $\pcF_{i}$ and $[-u,u]$.
Because of this step delay, the proximal term
plays the desired role for \eqref{UkBounded}
to hold under $\{U^{k-1}_{i,l}<0$, $U^{k}_{i,l}<0\}$
owing to the continuity of $U_{i,l}$ defined by \eqref{UkilDefn}.
(One may remove the zero branch from \eqref{proxXNormOcoefkilDefn} too.)
The term $P^{k}$ helps to counterbalance
 $U^{k}_{i,l}<0$ under $U^{k-1}_{i,l}\geq 0$
 in order to meet \eqref{UkBounded}. Specifically,
 while the number of $\pcF(Z)\leq 0$ is fixed,
 the size of $N$ (the number of $X_i-Z=0$) and the presence of
$\{\pcG(Z)\leq 0$, $\pcH(Z)=0\}$
(and their possibly redundant use in different $i$-CBs if necessary) promote
the magnitude of $P^{k}>0$, and therefore, $J^{k}>0$.
Here, we face an issue of balance:
 Big $N$ and redundant application of $\{\pcG(Z)\leq 0$, $\pcH(Z)=0\}$
increase the computational size of the algorithm;
however, they may lead to the possibility of
$\proxXNormOcoef^{k}_{i,l}=0$, and this makes it simpler to solve
  \eqref{ACPGSProximalXilA}.
The quadratic convexity of $\pcF_{i}$ yields
\begin{align}
&
\pcF_{i}(X^{k+1,k}_{i,l})-\pcF_{i}(X^{k+1,k+1}_{i,l})
    -(X^{k}_{i,l}-X^{k+1}_{i,l})^{T}\,\nabla_{X_{i,l}}\pcF_{i}(X^{k+1,k+1}_{i,l})
\notag\\
=\,&
\frac{1}{2}(X^{k}_{i,l}-X^{k+1}_{i,l})^{T}\,
        \nabla_{X_{i,l}}\nabla_{X_{i,l}}\pcF_{i}(X^{k+1,k+1}_{i,l})\,
              (X^{k}_{i,l}-X^{k+1}_{i,l}),
    \ \ \nabla_{X_{i,l}}\nabla_{X_{i,l}}\pcF_{i}(X^{k+1,k+1}_{i,l})\succeq 0,
\label{Norm1ProximalLimtB}
\end{align}
and
\begin{align}
&
 U^{k}_{i,l}=
 \frac{1}{2}\big\langle \dvF^{k}_{i}+\rho_i\big(\pcF_{i}(X^{k+1,k+1}_{i,l})+\svF^{k}_{i}\big),
                (X^{k}_{i,l}-X^{k+1}_{i,l})^{T}\,
        \nabla_{X_{i,l}}\nabla_{X_{i,l}}\pcF_{i}(X^{k+1,k+1}_{i,l})\,
              (X^{k}_{i,l}-X^{k+1}_{i,l})
            \big\rangle.
\label{Ukil}
\end{align}
Expression~\eqref{Ukil} indicates how to promote $U^{k}_{i,l}\geq 0$
with initial values $\{\dvF^{0}_{i}$,
$\pcF_{i}(X^{0}_{i})+\svF^{0}_{i}\}$ being greatly or moderately positive
and other initial and parameter values adequate.
It follows from \eqref{GlobalInequalitydvk},
 \eqref{GlobalInequalitydvSum2K}, and \eqref{UkBounded} that
the sequence $\{L^{k}\}_{k}$ decreases monotonically
and is bounded from above by $L^{0}$.
In the special case of linear programming,
$\proxXNormOcoef^{k}_{i,l}=0$ for all $k$
and $\{L^{k}\}_{k}$ decreases monotonically.

Secondly, the sequence $\{L^{k}\}_{k}$ is bounded from below whose lowest value possible
but unlikely is estimated as follows.
Consider the paired sum of the dual and penalty terms,
$\dvF^{k}_{i,j} (\pcF_{i}(X^{k+1,+}_{i,l})+\svF^{k}_{i})_{j}
    +(\rho_{i}/2)\vert (\pcF_{i}(X^{k+1,+}_{i,l})+\svF^{k}_{i})_{j}\vert^2$
    in $L_{i,l}$ whose absolute minimum is achieved at
    $\{\dvF^{k}_{i,j}=\udvF_{\mu_{i},j}$,
    $(\pcF_{i}(X^{k+1,+}_{i,l})+\svF^{k}_{i})_{j}=-\udvF_{\mu_{i},j}/\rho_{i}\}$;
    summation of such minima of all the paired sums in $L_{i,M}$
    for all $i$ yields the lowest possible bound of $L^{k}$.
    Therefore,
$\lim_{k\rightarrow\infty}L^{k}$ exists.
To make the present algorithm work, we need to raise the lower bound to
$\overline{\lim}_{k'\rightarrow\infty}\sum_{i} f(X^{k'}_{i})<L^{k}$ $\forall k$,
this is to be achieved operationally
through adequate parameter and initial values.

Thirdly,
the $X^{k+1}_{i,l}$-update, \eqref{ACPGSProximalXilA}
involves the couplings among all the extended component sequences
through the paired sums of the dual and penalty terms like
$\dvF^{k}_{i,j} (\pcF_{i}(X^{k+1,+}_{i,l})+\svF^{k}_{i})_{j}
    +(\rho_{i}/2)\vert (\pcF_{i}(X^{k+1,+}_{i,l})+\svF^{k}_{i})_{j}\vert^2$.
To place all the penalty terms on an equal footing,
we introduce Assumption~\ref{AssumptionSetI}(A)(B)(C).
It is expected that the $X^{k+1}_{i,l}$-update from
such a preconditioning treatment
 does not bias toward minimizing any specific
 paired sum overall,
 which is illustrated mathematically as follows:
\begin{align}
 \phi_1(x)+\phi_2(x)\geq  \phi_1(x^{\ast})+\phi_2(x^{\ast}),\
 \phi_1(x)\geq \phi_1(x^{\ast}_1),\
 \phi_2(x)\geq \phi_2(x^{\ast}_2)\ \ \forall x;\ \
 \phi_1(x^{\ast})\geq \phi_1(x^{\ast}_1),\
 \phi_2(x^{\ast})\geq \phi_2(x^{\ast}_2).
 \label{CouplingMeanProperty}
\end{align}
This coupling mean property regulates
the behaviors of $\{X^{k}_{i}\}_{k}$ directly
and the extended component sequences,
along with the other update rules;
it coordinates the evolution paces of the extended component sequences
such that no one changes much faster and toward extremely negative
on the whole;
specifically, an extended component element generated tends to be less negative,
if it is negative.
Consequently, it is expected to help prevent the dual component sequences
from being trapped in their NUBs
and to raise the lower bound of $\{L^{k}\}_{k}$ to
the one stated above.

Fourthly,
we adopt a modified version of
 Proposition~2.1 of \cite{Mangasarian1984} as Lemma~\ref{MangasarianLemma}
 to address the issues regarding the value of $\rho_{i}$
 and its role to affect the behavior of the extended sequences.
\begin{lemma}
Consider
\begin{align*}
  \min_{x}\big\{g(x)+ \alpha   Q(x)+\Indicator_{[l_{x}, u_{x}]}(x)\big\},
\end{align*}
where $g$ is convex, $Q$ convex and nonnegative,
and $\alpha$ is a positive constant.
Let $x_i$ be the solution for $\alpha_i$
with $\alpha_1<\alpha_2$.
Then,
\begin{align}
 g(x_1)\leq g(x_2),\
 Q(x_2)\leq Q(x_1),\
 g(x_1)+ \alpha_1 Q(x_1)\leq g(x_2)+ \alpha_2 Q(x_2).
 \label{MangasarianIneaualities}
\end{align}
\label{MangasarianLemma}
\end{lemma}
\vspace*{-4mm}
To apply Lemma~\ref{MangasarianLemma} to \eqref{ACPGSProximalXilA},
\eqref{ACPGSProximalZl},
and \eqref{ACPGSProximalsvZpiA} through \eqref{ACPGSProximalsvHniA},
we take
$\{\proxXNormTcoef^{k}_{i,l}$,
$\svZpcoef^{k}_{i}$,
$\svZncoef^{k}_{i}$,
$\svFcoef^{k}_{i}$,
$\svGcoef^{k}_{i}$,
$\svHpcoef^{k}_{i}$,
$\svHncoef^{k}_{i}\}\propto \rho_{i}\propto \rho_{1}$,
$\tau^{k}\propto \rho_{1}$,
and $\alpha\propto \rho_{1}/2$;
the forms of $g$ and $Q$ can be identified easily.
It follows that, on the whole,
greater $\rho_{i}$ tends to yield smaller
$\{\Vert \eZp^{k+1}_{i} \Vert$,
$\Vert \eZn^{k+1}_{i} \Vert$,
$\Vert \eF^{k+1}_{i} \Vert$,
$\Vert \eG^{k+1}_{i} \Vert$,
$\Vert \eHp^{k+1}_{i} \Vert$,
$\Vert \eHn^{k+1}_{i} \Vert\}$
and greater
$\big\{\s\sum_{i=1}^{N}\langle  \dvZp^{k}_{i}, \eZp^{k+1}_{i} \rangle$
$+\sum_{i=1}^{N}\langle  \dvZn^{k}_{i}, \eZn^{k+1}_{i} \rangle$,
$\langle  \dvF^{k}_{i}, \eF^{k+1}_{i}\rangle$,
$\langle  \dvG^{k}_{i}, \eG^{k+1}_{i}\rangle$,
$\langle  \dvHp^{k}_{i},\eHp^{k+1}_{i}\rangle
+\langle  \dvHn^{k}_{i},\eHn^{k+1}_{i}\rangle\big\}$.
Combination of the two implies that
\begin{align}
&
 \text{on the whole, greater $\rho_{i}$ tends to make}\
 \big\{\eZp^{k+1}_{i,j},
\eZn^{k+1}_{i,j},
\eF^{k+1}_{i,j},
\eG^{k+1}_{i,j},
\eHp^{k+1}_{i,j},
\eHn^{k+1}_{i,j}\big\} \
 \text{less in magnitude}
 \notag\\&
 \text{(and less negative if one is negative).}
 \label{EffectsOfGreaterPenalty}
\end{align}
These consequences are desirable, concerning
whether $\{L^{k}\}_{k}$ is bounded from below by the raised lower bound
and whether the dual component sequences are trapped
in their NUBs or not,
as illustrated by the case of $\eF^{k+1}_{i,j}<0$:
$\eF^{k+1}_{i,j}$ is less in magnitude and less negative,
 consequently, $\dvF^{k+1}_{i,j}$ has a smaller increase
 and has less chance to be trapped in its NUB
 according to
 \eqref{DualUpdateRules} and \eqref{DualUpdateRulesdvEFGBoundsj}.
Further, greater $\rho_{i}$ yields greater $L^{0}$, as desired.
A negative impact of too great a $\rho_{i}$ computationally is
that the evolution of $\{L^{k}\}_{k}$ is too slow,
as suggested by Lemma~\ref{MangasarianLemma}.
The question is how great $\rho_{i}$ needs to be,
in order to make the algorithm produce optimal limits in practice.

Fifthly,
the 2-norm proximal terms in the updates~\eqref{ACPGSProximalXilA},
\eqref{ACPGSProximalZl}, and
\eqref{ACPGSProximalsvZpiA} through \eqref{ACPGSProximalsvHniA}
provide flexibility for initialization of the primal sequences,
especially about the relationship between $X^{0}_{i}$ and $Z^{0}$
like $X^{0}_{i}=Z^{0}$ and obtaining great $L^{0}$.
Also, the proximal terms
control partially the evolution paces of the sequences and $\{L^{k}\}_{k}$,
serving to uphold the global
inequalities~\eqref{GlobalInequalitydvk}
and \eqref{GlobalInequalitydvSum2K}.
Motivated by the structures of the update rules
\eqref{SolvingXil} through \eqref{SolvingsvHnk},
the proximal control parameters are fixed through
\begin{align}
&
\proxXNormTcoef^{0}_{i,l}
=\proxXpropto_{i,l}\,
\max\Big\{
\Big\Vert
(f_{l})^{T}
+\dvZp^{0}_{i,l}-\dvZn^{0}_{i,l}
+\big\langle\dvF^{0}_{i},\nabla_{X_{i,l}}\pcF_{i}(X^{0}_{i})\big\rangle
+(\pcG_{i,l})^{T}\dvG^{0}_{i}
+(\pcH_{i,l})^{T} \big(\dvHp^{0}_{i}-\dvHn^{0}_{i}\big)
\notag\\[-2pt]&\hspace{31mm}
+\rho_{i}\Big(
     2\big(X^{0}_{i,l}-Z^{0}_{l}\big)
    +\svZp^{0}_{i,l}-\svZn^{0}_{i,l}
    +\big\langle \pcF_{i}(X^{0}_{i})+\svF^{0}_{i},
                  \nabla_{X_{i,l}}\pcF_{i}(X^{0}_{i})\big\rangle
    \notag\\[-2pt]&\hskip41mm
    +(\pcG_{i,l})^{T}\big(\pcG_{i}(X^{0}_{i})+\svG^{0}_{i}\big)
    +(\pcH_{i,l})^{T}\big(2\pcH_{i}(X^{0}_{i})
                          +\svHp^{0}_{i}-\svHn^{0}_{i}\big)
    \Big)
\Big\Vert_{\infty}
,\,
    6\rho_{i}\Big\},
\notag\\[-0pt]&
\proxXNormTcoef^{k}_{i,l}
=\proxXpropto_{i,l}\,
\max\Big\{
\Big\Vert (f_{l})^{T}
+\dvZp^{k-1}_{i,l}-\dvZn^{k-1}_{i,l}
+\big\langle\dvF^{k-1}_{i},\nabla_{X_{i,l}}\pcF_{i}(X^{k,k}_{i,l})\big\rangle
+(\pcG_{i,l})^{T}\dvG^{k-1}_{i}
+(\pcH_{i,l})^{T} \big(\dvHp^{k-1}_{i}-\dvHn^{k-1}_{i}\big)
\notag\\[-2pt]&\hspace{31mm}
+\rho_{i}\Big(
     2\big(X^{k}_{i,l}-Z^{k-1}_{l}\big)
    +\svZp^{k-1}_{i,l}-\svZn^{k-1}_{i,l}
    +\big\langle \pcF_{i}(X^{k,k}_{i,l})+\svF^{k-1}_{i},
                  \nabla_{X_{i,l}}\pcF_{i}(X^{k,k}_{i,l})\big\rangle
    \notag\\[-2pt]&\hskip41mm
    +(\pcG_{i,l})^{T}\big(\pcG_{i}(X^{k,k}_{i,l})+\svG^{k-1}_{i}\big)
    +(\pcH_{i,l})^{T}\big(2\pcH_{i}(X^{k,k}_{i,l})
                          +\svHp^{k-1}_{i}-\svHn^{k-1}_{i}\big)
    \Big)\Big\Vert_{\infty},
\notag\\[-2pt]&\hskip28mm
    \proxXproptorhoi\,\rho_{i}\Big\}
    \ \ \forall k\geq 1;
\label{ProximalControlParameter2sigmaA}
\end{align}
\vspace{-6mm}
\begin{align}
&
\tau^{0}
=
\proxZpropto
\max\Big\{\Big\Vert
    \frac{1}{N}\sum_{i=1}^{N}\Big(
         \dvZn^{0}_{i,l}-\dvZp^{0}_{i,l}
        +\rho_{i}\big(
                      2(Z^{0}_{l}-X^{0}_{i,l})-\svZp^{0}_{i,l}+\svZn^{0}_{i,l}
                                    \big)
        \Big)
    \Big\Vert_{\infty},\,
      \frac{1}{N}\sum 2\rho_{i}\Big\},
\notag\\&
\tau^{k}
=
\proxZpropto
\max\Big\{\Big\Vert
    \frac{1}{N}\sum_{i}\Big(
         \dvZn^{k-1}_{i,l}-\dvZp^{k-1}_{i,l}
        +\rho_{i}\big(
                      2(Z^{k}_{l}-X^{k}_{i,l})-\svZp^{k-1}_{i,l}+\svZn^{k-1}_{i,l}
                                    \big)
        \Big)
    \Big\Vert_{\infty},\,
      \frac{1}{N}\sum\proxZproptorhoi\rho_{i}\Big\}\ \ \forall k\geq 1;
\label{ProximalControlParametertaukA}
\end{align}
\vspace{-6mm}
\begin{align}
&
\svZpcoef^{0}_{i}
=\proxsvpXpropto_{i}
\max\Big\{\big\Vert
      \dvZp^{0}_{i}
      +\rho_{i}(X^{0}_{i}-Z^{0}_{}+\svZp^{0}_{i})
  \big\Vert_{\infty},\,
      \rho_{i}\Big\},
\notag\\&
\svZpcoef^{k}_{i}
=\proxsvpXpropto_{i}
\max\Big\{\Big\Vert
      \dvZp^{k-1}_{i}
      +\rho_{i}(X^{k}_{i}-Z^{k}_{}+\svZp^{k}_{i})
  \Big\Vert_{\infty},\,
      \proxsvpXproptorhoi\rho_{i}\Big\} \ \ \forall k\geq 1;
\label{ProximalControlParameterZpgammaA}
\end{align}
\vspace{-6mm}
\begin{align}
&
\svZncoef^{0}_{i}
=\proxsvnXpropto_{i}
\max\Big\{\big\Vert
        \dvZn^{0}_{i}
          +\rho_{i}( Z^{0}_{}-X^{0}_{i}+\svZn^{0}_{i} )
    \big\Vert_{\infty},\,
      \rho_{i}\Big\},
\notag\\&
\svZncoef^{k}_{i}
=\proxsvnXpropto_{i}
\max\Big\{\big\Vert
        \dvZn^{k-1}_{i}
          +\rho_{i}( Z^{k}_{}-X^{k}_{i}+\svZn^{k}_{i} )
    \big\Vert_{\infty},\,
      \proxsvnXproptorhoi\rho_{i}\Big\}\ \ \forall k\geq 1;
\label{ProximalControlParameterZngammaA}
\end{align}
\vspace{-6mm}
\begin{align}
&
\svFcoef^{0}_{i}
=\proxsvFpropto_{i}
\max\Big\{\big\Vert
        \dvF^{0}_{i}
        +\rho_{i}( \pcF_{i}(X^{0}_{i})+\svF^{0}_{i} )
      \big\Vert_{\infty},\,
      \rho_{i}\Big\},
\notag\\&
\svFcoef^{k}_{i}
=\proxsvFpropto_{i}
\max\Big\{\big\Vert
        \dvF^{k-1}_{i}
        +\rho_{i}( \pcF_{i}(X^{k}_{i})+\svF^{k}_{i} )
      \big\Vert_{\infty},\,
      \proxsvFproptorhoi\rho_{i}\Big\}\ \ \forall k\geq 1;
\label{ProximalControlParameterFgammaA}
\end{align}
\vspace{-6mm}
\begin{align}
&
\svGcoef^{0}_{i}
=\proxsvGpropto_{i}
\max\Big\{\big\Vert
      \dvG^{0}_{i}
+\rho_{i}( \pcG_{i}(X^{0}_{i})+\svG^{0}_{i} )\big\Vert_{\infty},\,
      \rho_{i}\Big\},
\notag\\&
\svGcoef^{k}_{i}
=\proxsvGpropto_{i}
\max\Big\{\big\Vert
      \dvG^{k-1}_{i}
+\rho_{i}( \pcG_{i}(X^{k}_{i})+\svG^{k}_{i} )\big\Vert_{\infty},\,
      \proxsvGproptorhoi\rho_{i}\Big\}\ \ \forall k\geq 1;
\label{ProximalControlParameterGgammaA}
\end{align}
\vspace{-6mm}
\begin{align}
&
\svHpcoef^{0}_{i}
=\proxsvpHpropto_{i}
\max\Big\{\big\Vert
          \dvHp^{0}_{i}
        +\rho_{i}( \pcH_{i}(X^{0}_{i})+\svHp^{0}_{i} )
       \big\Vert_{\infty},\,
      \rho_{i}\Big\},
\notag\\&
\svHpcoef^{k}_{i}
=\proxsvpHpropto_{i}
\max\Big\{\big\Vert
          \dvHp^{k-1}_{i}
        +\rho_{i}( \pcH_{i}(X^{k}_{i})+\svHp^{k}_{i} )
       \big\Vert_{\infty},\,
      \proxsvpHproptorhoi\rho_{i}\Big\}\ \ \forall k\geq 1;
\label{ProximalControlParameterpHgammaA}
\end{align}
and
\begin{align}
&
\svHncoef^{0}_{i}
=\proxsvnHpropto_{i}
\max\Big\{  \big\Vert
          \dvHn^{0}_{i}
        +\rho_{i}( -\pcH_{i}(X^{0}_{i})+\svHn^{0}_{i} )
        \big\Vert_{\infty},\,
      \rho_{i}\Big\},
\notag\\&
\svHncoef^{k}_{i}
=\proxsvnHpropto_{i}
\max\Big\{ \big\Vert
          \dvHn^{k-1}_{i}
        +\rho_{i}( -\pcH_{i}(X^{k}_{i})+\svHn^{k}_{i} )
        \big\Vert_{\infty},\,
      \proxsvnHproptorhoi\rho_{i}\Big\}\ \ \forall k\geq 1.
\label{ProximalControlParameternHgammaA}
\end{align}
Here,
$X^{k,k}_{i,l}=X^{k,k,k-1}_{i,l}$.
The 2-norm proximal proportionality coefficients
$\{\proxXpropto_{i,l}$,\,$\proxZpropto$,\,$\proxsvpXpropto_{i}$,\,$\proxsvnXpropto_{i}$,\,$\proxsvFpropto_{i}$,\,$\proxsvGpropto_{i}$,\,$\proxsvpHpropto_{i}$,\,$\proxsvnHpropto_{i}\}$
and
$\{\proxXproptorhoi$,\,$\proxZproptorhoi$,\,$\proxsvpXproptorhoi$,
$\proxsvnXproptorhoi$,\,$\proxsvFproptorhoi$,\,$\proxsvGproptorhoi$,\,$\proxsvpHproptorhoi$,\,$\proxsvnHproptorhoi\}$
are positive constants
to be specified.
The structures of \eqref{ProximalControlParameter2sigmaA}
are adopted on the basis of the following considerations:
(A)
The magnitudes of the gradient-related terms in \eqref{SolvingXil}
are unknown (and can be great) and there is the need to control the value
of $X^{k+1}_{i,l}-X^{k}_{i,l}$ via the value of
$\proxXNormTcoef^{k}_{i,l}$
and this value should be self-adaptive to the sequences.
(B)
The data are available at the $k$-th iteration.
(C)
(\ref{ProximalControlParameter2sigmaA}$)_2$
comes from the one-step delay of
\eqref{SolvingXil} whose subgradient terms are neglected.
There is also the possibility to use the newest dual data available, such as
$X^{k,k,k}_{i,l}$,
 $\dvZp^{k}_{i,l}-\dvZn^{k}_{i,l}$,
$\dvHp^{k}_{i}-\dvHn^{k}_{i}$, $\dvF^{k}_{i}$,
 $\dvG^{k}_{i}$, etc.
(D)
At great $k$,
$\proxXproptorhoi\,\rho_{i}$ is the dominant quantity in the $\max$-operation,
directly resulting in the application of \eqref{EffectsOfGreaterPenalty}.
The maximum pace of the component sequences of $\{X^{k}_{i,l}\}_{k}$ is
$\Vert X^{k+1}_{i,l}-X^{k}_{i,l}\Vert_{\infty}$ $\approx (\proxXpropto_{i,l})^{-1}$.
Similarly,
the structures of \eqref{ProximalControlParametertaukA}
through \eqref{ProximalControlParameternHgammaA} are justified based on
the forms of
\eqref{SolvingZk} through \eqref{SolvingsvHnk}, respectively.
It is expected that
adequate values of $\{\proxXpropto_{i,l}$,\,$\proxZpropto$,\,$\proxsvpXpropto_{i}$,\,$\proxsvnXpropto_{i}$,\,$\proxsvFpropto_{i}$,\,$\proxsvGpropto_{i}$,\,$\proxsvpHpropto_{i}$,\,$\proxsvnHpropto_{i}$,
$\rho_{i}\}$
and relatively small values of
$\{\dvZpcoef_{i}$,\,$\dvZncoef_{i}$,\,$\dvFcoef_{i}$,\,$\dvGcoef_{i}$,\,$\dvHpcoef_{i}$,\,$\dvHncoef_{i}\}$
make the the quantities at the $k$-th and the $(k+1)$-th iterations
close to each other,
and the paces of the primal and dual sequences are well regulated
by the proximal control parameter models above.

Finally,
to provide the conditions that
all the dual component sequences,
$\{\dvZp^{k}_{i,j}$,\,$\dvZn^{k}_{i,j}$,\,$\dvF^{k}_{i,j}$,\,$\dvG^{k}_{i,j}$,\,$\dvHp^{k}_{i,j}$,\,$\dvHn^{k}_{i,j}\}_{k}$
are not trapped in the NUBs and more,
we resort to the following specific measures
on the basis of the preceding discussions.
(A)
It is tentatively proposed to take
\begin{align}
&
\big\{\s
\proxXpropto_{i,l},\proxZpropto,\proxsvpXpropto_{i},\proxsvnXpropto_{i},
\proxsvFpropto_{i},\proxsvGpropto_{i},\proxsvpHpropto_{i},\proxsvnHpropto_{i}
\big\}
\approx 10^{3},\
  \proxXproptorhoi=6,\
  \proxZproptorhoi=2,\
\proxsvpXproptorhoi=\proxsvnXproptorhoi=\proxsvFproptorhoi=\proxsvGproptorhoi
  =\proxsvpHproptorhoi=\proxsvnHproptorhoi=1.
\label{Norm2ProptCoeffsDvControlParameters}
\end{align}
The values ensure that
the maximum paces of the primal component sequences
are about or less than $10^{-3}$, which is compatible with the preconditioned ranges and domains
of the primal variables in Assumption~\ref{AssumptionSetI}.
 As indicated by \eqref{EffectsOfGreaterPenalty},
a value of $\rho_{i}$, say
 \begin{align}
 \rho_{i}\in[10^{2}, 10^{3}],
 \label{PenaltyParameter}
 \end{align}
tends to make all the extended component sequences
relatively small and less negative if one is negative, at great $k$
where the $\proxXproptorhoi\,\rho_{i}$ terms and so on
dominate the $\max$-operations
in \eqref{ProximalControlParameter2sigmaA}
through \eqref{ProximalControlParameternHgammaA}.
(B)
On the one hand, the differences,
$\{\dvZp^{k}_{i}-\dvZn^{k}_{i},\dvHp^{k}_{i}-\dvHn^{k}_{i}\}$
count in the $X^{k+1}_{i,l}$-update~\eqref{SolvingXil}
and the $Z^{k+1}_{l}$-update~\eqref{SolvingZk};
it suggests the choice of greatly positive values for
$\{\dvZp^{0}_{i}$,\,$\dvZn^{0}_{i}$,\,$\dvHp^{0}_{i}$,\,$\dvHn^{0}_{i}\}$,
which makes $L^{0}$ great.
On the other hand,
$\{\dvF^{\infty}_{i,j}$,\,$\dvG^{\infty}_{i,j}\}=0$ hold
for almost all $i$, $j$,
except the few that are active;
these limiting conditions suggest
 moderately positive values for $\{\dvF^{0}_{i}$,\,$\dvG^{0}_{i}\}$, which may help
the non-active ones of $\{\dvF^{k}_{i,j},\dvG^{k}_{i,j}\}_{k}$ approaching zero
within a reasonable number of iterations.
The following initial values and the values for the dual scalar constants
of \eqref{DualUpdateRulesdvEFGBoundsj} are adopted,
\begin{align}
&
X^{0}_{i}=Z^{0},\
\svZp^{0}_{i}=\svZn^{0}_{i}=2,\
\svHp^{0}_{i}=\svHn^{0}_{i}=2,\
\svF^{0}_{i}=1,\
\svG^{0}_{i}=1,\
\dvZp^{0}_{i}=\dvZn^{0}_{i}=\rho_{i}\svZp^{0}_{i},\
\dvHp^{0}_{i}=\rho_{i}\eHp^{0}_{i},
\notag\\&
\dvHn^{0}_{i}=\rho_{i}\eHn^{0}_{i},\
\dvF^{0}_{i}=10\eF^{0}_{i},\
\dvG^{0}_{i}=10\eG^{0}_{i},\
\big\{
\dvZpcoef_{i},\dvZncoef_{i},\dvFcoef_{i},\dvGcoef_{i},
\dvHpcoef_{i},\dvHncoef_{i}\big\}=10^{-2}.
\label{InitialParameterValuesA}
\end{align}
They help to generate great $L^{0}$,
 to raise the lower bound for $\{L^{k}\}_{k}$ to that stated above,
and to uphold \eqref{GlobalInequalitydvSum2K} with meaningful consequences
from $\sum_{k=0}^{\infty}D^{k}<\infty$;
together with \eqref{Norm2ProptCoeffsDvControlParameters},
\eqref{PenaltyParameter},
\eqref{CouplingMeanProperty}, and \eqref{EffectsOfGreaterPenalty},
they tend to make $\{\eZp^{k}_{i,j}$,
$\eZn^{k}_{i,j}$,
$\eF^{k}_{i,j}$,
$\eG^{k}_{i,j}$,
$\eHp^{k}_{i,j}$,
$\eHn^{k}_{i,j}\}$
non-negative, less negative, or close to zero.
(C)
On the basis of the forms of \eqref{SolvingsvZpk}
through \eqref{SolvingsvHnk},
the upper bounds on the dual variables are taken as
\begin{align}
\big\{\s\udvZp_{\mu_{i}},\udvZn_{\mu_{i}},\udvF_{\mu_{i}},\udvG_{\mu_{i}},
\udvHp_{\mu_{i}},\udvHn_{\mu_{i}}\big\}
\propto \rho_{i}
\big\{\usvZp_{Y_{i}},\usvZn_{Y_{i}},\usvF_{Y_{i}},\usvG_{Y_{i}},\usvHp_{Y_{i}},
        \usvHn_{Y_{i}}\big\},
\label{UpperBoundsDualVariables}
\end{align}
where the dual upper bound proportionality coefficients take a value of 5, say,
such that the bounds are great to help prevent the dual component sequences from
being trapped in the NUBs.
(D)
Together with the couplings among all the primal and dual
through the update rules and $L^{k+1}<L^{k}$,
 the collective impacts of the above-listed measures
 are expected to make the extended and dual component sequences
 coevolve rather slowly and interacting with each other effectively
 such that all the dual component sequences are not trapped
  in the NUBs and
all the component sequences,
 $\{\dvZp^{k}_{i,j}$, $\dvZn^{k}_{i,j}$,\,$\dvHp^{k}_{i,j}$,\,$\dvHn^{k}_{i,j}\}_{k}$
 are not trapped in their NLBs.
\begin{assumption}
$\qquad$
\begin{enumerate}
\item[(A)]
The sequence $\{L^{k}\}_{k}$ decreases monotonically, bounded from below.
All the dual component sequences,
$\{\dvZp^{k}_{i,j}$,
$\dvZn^{k}_{i,j}$,\,$\dvF^{k}_{i,j}$,\,$\dvG^{k}_{i,j}$,\,$\dvHp^{k}_{i,j}$,\,$\dvHn^{k}_{i,j}\}_{k}$
are not trapped in their NUBs;
all the dual component sequences,
$\{\dvZp^{k}_{i,j}$,\,$\dvZn^{k}_{i,j}$, $\dvHp^{k}_{i,j}$,\,$\dvHn^{k}_{i,j}\}_{k}$
 are not trapped in their NLBs;
all paired dual-slack component sequences, such as $\{\dvF^{k}_{i,j},\svF^{k}_{i,j}\}_{k}$ and
$\{\dvG^{k}_{i,j},\svG^{k}_{i,j}\}_{k}$ are not trapped in the NLB and at zero simultaneously.
To this end, a range of choices
for the parameter and initial values are supposed to exist.
Along with Assumption~\ref{AssumptionSetI},
these conditions are
called the feasibility conditions,
and the algorithm is said to be feasible,
if the feasibility conditions hold.

\item[(B)]
 The upper bounds for the slack variable sequences
   in \eqref{ACPGSProximalsvZpiA}
   through \eqref{ACPGSProximalsvHniA}
   are out of reach, to be specified later.

\item[(C)]
  All the parameters are positive and finite.
  Except those assigned explicitly by \eqref{DualUpdateRulesdvEFGBoundsj}
  and \eqref{proxXNormOcoefkilDefn},
  the parameter sequences are bounded from above and below
  by positive finite values.

\item[(D)]
Each and every primal component sequence
has at most finitely many cluster points,
which is viewed as part of the feasibility conditions.
(This condition can be augmented or replaced in the following manner.
Consider a primal component sequence $\{y^{k}\}_{k\geq K}$
which has a subsequence $\{y^{k_j}\}_{j}$ convergent to zero.
Motivated by the structure of $P^{k}$ and its role in the algorithm,
an index difference sequence,
$\{k_{j+1}-k_{j}\}_{j}$ is introduced.
The index sequence is supposedly bounded from above by $\Delta k$, say.)
\end{enumerate}
\label{AssumptionSetII}
\end{assumption}
\vspace*{-4mm}
\begin{lemma} 
The sequences,
\begin{align*}
 \big\{&
   X^{k}_{i}, Z^{k}, \svZp^{k}_{i}, \svZn^{k}_{i}, \svF^{k}_{i},
   \svG^{k}_{i}, \svHp^{k}_{i}, \svHn^{k}_{i},
  \dvZp^{k}_{i}, \dvZn^{k}_{i}, \dvF^{k}_{i},
           \dvG^{k}_{i}, \dvHp^{k}_{i}, \dvHn^{k}_{i},
    X^{k+1}_{i}, Z^{k+1}, \svZp^{k+1}_{i}, \svZn^{k+1}_{i},
  \svF^{k+1}_{i},
   \svG^{k+1}_{i},
   \notag\\&
   \svHp^{k+1}_{i}, \svHn^{k+1}_{i},
  \dvZp^{k+1}_{i}, \dvZn^{k+1}_{i}, \dvF^{k+1}_{i},
  \dvG^{k+1}_{i},
  \dvHp^{k+1}_{i}, \dvHn^{k+1}_{i},
  \
    i=1,\ldots, N \big\}_{k}
\end{align*}
are bounded;
there exist convergent subsequences, 
\begin{align*}
 \big\{&
   X^{k_{j}}_{i}, Z^{k_{j}}, \svZp^{k_{j}}_{i}, \svZn^{k_{j}}_{i}, \svF^{k_{j}}_{i},
   \svG^{k_{j}}_{i}, \svHp^{k_{j}}_{i}, \svHn^{k_{j}}_{i},
  \dvZp^{k_{j}}_{i}, \dvZn^{k_{j}}_{i}, \dvF^{k_{j}}_{i},
           \dvG^{k_{j}}_{i}, \dvHp^{k_{j}}_{i}, \dvHn^{k_{j}}_{i},
         X^{k_{j}+1}_{i}, Z^{k_{j}+1},
  \svZp^{k_{j}+1}_{i}, 
\notag\\&
  \svZn^{k_{j}+1}_{i},
 \svF^{k_{j}+1}_{i},
   \svG^{k_{j}+1}_{i}, \svHp^{k_{j}+1}_{i}, \svHn^{k_{j}+1}_{i},
  \dvZp^{k_{j}+1}_{i}, \dvZn^{k_{j}+1}_{i}, \dvF^{k_{j}+1}_{i},
  \dvG^{k_{j}+1}_{i},
  \dvHp^{k_{j}+1}_{i},
 \dvHn^{k_{j}+1}_{i},
 \ 
    i=1,\ldots, N \big\}_{j}
\end{align*}
and cluster points such that for all $i$,
\begin{align}
&
\lim_{j\rightarrow\infty} X^{k_j}_{i}=X^{\infty}_{i},\
\lim_{j\rightarrow\infty} X^{k_j+1}_{i}=X^{\infty+}_{i},\
\lim_{j\rightarrow\infty} Z^{k_j}=Z^{\infty},\
\lim_{j\rightarrow\infty} Z^{k_j+1}=Z^{\infty+},\
\lim_{j\rightarrow\infty} \svZp^{k_j}_{i}=\svZp^{\infty}_{i},
\notag\\&
\lim_{j\rightarrow\infty} \svZp^{k_j+1}_{i}=\svZp^{\infty+}_{i},\
\lim_{j\rightarrow\infty} \svZn^{k_j}_{i}=\svZn^{\infty}_{i},\
\lim_{j\rightarrow\infty} \svZn^{k_j+1}_{i}=\svZn^{\infty+}_{i},\
\lim_{j\rightarrow\infty} \svF^{k_j}_{i}=\svF^{\infty}_{i},\
\lim_{j\rightarrow\infty} \svF^{k_j+1}_{i}=\svF^{\infty+}_{i},
\notag\\&
\lim_{j\rightarrow\infty} \svG^{k_j}_{i}=\svG^{\infty}_{i},\
\lim_{j\rightarrow\infty} \svG^{k_j+1}_{i}=\svG^{\infty+}_{i},\
\lim_{j\rightarrow\infty} \svHp^{k_j}_{i}=\svHp^{\infty}_{i},\
\lim_{j\rightarrow\infty} \svHp^{k_j+1}_{i}=\svHp^{\infty+}_{i},\
\lim_{j\rightarrow\infty} \svHn^{k_j}_{i}=\svHn^{\infty}_{i},
\notag\\&
\lim_{j\rightarrow\infty} \svHn^{k_j+1}_{i}=\svHn^{\infty+}_{i},\
\lim_{j\rightarrow\infty} \dvZp^{k_j}_{i}=\dvZp^{\infty}_{i},\
\lim_{j\rightarrow\infty} \dvZp^{k_j+1}_{i}=\dvZp^{\infty+}_{i},\
\lim_{j\rightarrow\infty} \dvZn^{k_j}_{i}=\dvZn^{\infty}_{i},\
\lim_{j\rightarrow\infty} \dvZn^{k_j+1}_{i}=\dvZn^{\infty+}_{i},
\notag\\&
\lim_{j\rightarrow\infty} \dvF^{k_j}_{i}=\dvF^{\infty}_{i},\
\lim_{j\rightarrow\infty} \dvF^{k_j+1}_{i}=\dvF^{\infty+}_{i},\
\lim_{j\rightarrow\infty} \dvG^{k_j}_{i}=\dvG^{\infty}_{i},\ 
\lim_{j\rightarrow\infty} \dvG^{k_j+1}_{i}=\dvG^{\infty+}_{i},\
\lim_{j\rightarrow\infty} \dvHp^{k_j}_{i}=\dvHp^{\infty}_{i},
\notag\\&
\lim_{j\rightarrow\infty} \dvHp^{k_j+1}_{i}=\dvHp^{\infty+}_{i},\
\lim_{j\rightarrow\infty} \dvHn^{k_j}_{i}=\dvHn^{\infty}_{i},\ 
\lim_{j\rightarrow\infty} \dvHn^{k_j+1}_{i}=\dvHn^{\infty+}_{i}.
\end{align}
\label{PDAccumulationPointLemma}
\end{lemma}
\vspace*{-6mm}
\begin{lemma}
 The following sums are bounded.
 \vspace*{-0mm}
\begin{enumerate}
\item[(A)] $\forall i$,
\vspace*{-3mm}
\begin{align*} 
&
\sum_{k=1}^{\infty}\sum_{j}
  \dvZpcoef^{k}_{i,j}\big\vert \big(X^{k}_{i}-Z^{k}_{}+\svZp^{k}_{i}\big)_{j}\big\vert^2<\infty,\ \
\sum_{k=1}^{\infty}\sum_{j}
  \dvZncoef^{k}_{i,j}\big\vert \big(Z^{k}_{}-X^{k}_{i}+\svZn^{k}_{i}\big)_{j}\big\vert^2<\infty,
\notag\\&
\sum_{k=1}^{\infty}\sum_{j}
  \dvFcoef^{k}_{i,j}\big\vert \big(\pcF_{i}(X^{k}_{i})+\svF^{k}_{i}\big)_{j}\big\vert^2<\infty,\ \
\sum_{k=1}^{\infty}\sum_{j}
  \dvGcoef^{k}_{i,j}\big\vert \big(\pcG_{i}(X^{k}_{i})+\svG^{k}_{i}\big)_{j}\big\vert^2<\infty,
\notag\\&
\sum_{k=1}^{\infty}\sum_{j}
  \dvHpcoef^{k}_{i,j}\big\vert \big(\pcH_{i}(X^{k}_{i})+\svHp^{k}_{i}\big)_{j}\big\vert^2<\infty,\ \
\sum_{k=1}^{\infty}\sum_{j}
  \dvHncoef^{k}_{i,j}\big\vert \big(-\pcH_{i}(X^{k}_{i})+\svHn^{k}_{i}\big)_{j}\big\vert^2<\infty.
\end{align*}
\item[(B)] $\forall i$,
\vspace*{-3mm}
\begin{align*}
  &
\sum_{k=0}^{\infty}\big\Vert X^{k+1}_{i}-X^{k}_{i} \big\Vert^2<\infty, \ \
\sum_{k=0}^{\infty}\big\Vert Z^{k+1}_{}-Z^{k}_{}\big\Vert^2<\infty, \ \
\sum_{k=0}^{\infty}\big\Vert \svZp^{k+1}_{i}-\svZp^{k}_{i}\big\Vert^2<\infty, 
\notag\\&
\sum_{k=0}^{\infty}\big\Vert\svZn^{k+1}_{i}-\svZn^{k}_{i}\Vert^2<\infty,\ \
\sum_{k=0}^{\infty}\big\Vert \svF^{k+1}_{i}-\svF^{k}_{i} \big\Vert^2<\infty, \ \
\sum_{k=0}^{\infty}\big\Vert \svG^{k+1}_{i}-\svG^{k}_{i} \big\Vert^2<\infty, 
\notag\\&
\sum_{k=0}^{\infty}\big\Vert \svHp^{k+1}_{i}-\svHp^{k}_{i}\big\Vert^2<\infty, \ \
\sum_{k=0}^{\infty}\big\Vert \svHn^{k+1}_{i}-\svHn^{k}_{i} \big\Vert^2<\infty.
\end{align*}
\item[(C)] $\forall i$,
\vspace*{-3mm}
\begin{align*}
&
\sum_{k=0}^{\infty}\sum_{l}\big\Vert \pcF_{i}(X^{k+1,k+1}_{i,l})-\pcF_{i}(X^{k+1,k}_{i,l}) \big\Vert^2<\infty, \ \ 
\sum_{k=0}^{\infty}\sum_{l}\big\Vert\pcG_{i,l}(X^{k+1}_{i,l}-X^{k}_{i,l}) \big\Vert^2<\infty, 
\notag\\&
\sum_{k=0}^{\infty}\sum_{l}\big\Vert \pcH_{i,l}(X^{k+1}_{i,l}-X^{k}_{i,l}) \big\Vert^2<\infty.
\end{align*}
\end{enumerate}
\label{PDSumBoundedLemma}
\end{lemma}
\vspace*{-6mm}
\begin{proof}
 Take $K=k_j$ in Lemma~\ref{GlobalInequalitydvSum2KLemma};
 apply Lemma~\ref{PDAccumulationPointLemma}
 and Assumption~\ref{AssumptionSetII}(C).
\end{proof}
\vspace*{-3mm}
\begin{lemma}
The limits and stationarity conditions.
\vspace*{-1mm}
 \begin{enumerate}
 \item[(A)] 
 $\forall i$,
\vspace*{-1mm}
\begin{align}
&
\lim_{k\rightarrow\infty}\big(X^{k}_{i}-Z^{k}\big)=
\lim_{k\rightarrow\infty}\svZp^{k}_{i}=
\lim_{k\rightarrow\infty}\svZn^{k}_{i}=0,\ 
\lim_{k\rightarrow\infty}\big(\pcF_{i}(X^{k}_{i})+\svF^{k}_{i}\big)=0,\ 
\lim_{k\rightarrow\infty} \big\langle \dvF^{k}_{i},\pcF_{i}(X^{k}_{i})\big\rangle=0,\ 
\notag\\[0mm]&
\lim_{k\rightarrow\infty}\big(\pcG_{i}(X^{k}_{i})+\svG^{k}_{i}\big)=0,\ \
\lim_{k\rightarrow\infty} \big\langle \dvG^{k}_{i},\pcG_{i}(X^{k}_{i})\big\rangle=0,\ \
\lim_{k\rightarrow\infty}\pcH_{i}(X^{k}_{i})=
\lim_{k\rightarrow\infty}\svHp^{k}_{i}=
\lim_{k\rightarrow\infty}\svHn^{k}_{i}=0,
\notag\\[0mm]&
\lim_{k\rightarrow\infty}\big(\pcF_{i}(X^{k+1,k+1}_{i,l})-\pcF_{i}(X^{k+1,k}_{i,l}) \big)=0,\ \ 
\lim_{k\rightarrow\infty}\big(X^{k+1}_{i}-X^{k}_{i}\big)=
\lim_{k\rightarrow\infty}\big(Z^{k+1}-Z^{k}\big)=0,
\notag\\[0mm]&
\lim_{k\rightarrow\infty}\big(\svF^{k+1}_{i}-\svF^{k}_{i}\big)=0,\ \
\lim_{k\rightarrow\infty}\big(\svG^{k+1}_{i}-\svG^{k}_{i}\big)=0,\ \
X^{\infty+}_{i}=X^{\infty}_{i},\ \
Z^{\infty+}=Z^{\infty},\ \
\svF^{\infty+}_{i}=\svF^{\infty}_{i},
\notag\\[0mm]&
\svG^{\infty+}_{i}=\svG^{\infty}_{i},\ \
X^{\infty}_{i}-Z^{\infty}=
\svZp^{\infty}_{i}=
\svZn^{\infty}_{i}= 0,\ \
\dvZp^{\infty}_{i}\geq 0,\ \
\dvZn^{\infty}_{i}\geq 0,\ \
\pcF_{i}(X^{\infty}_{i})\leq 0,\ \
\svF^{\infty}_{i}\geq 0,
\notag\\[0mm]&
\dvF^{\infty}_{i}\geq 0,\ \
\big\langle \dvF^{\infty}_{i},\pcF_{i}(X^{\infty}_{i})\big\rangle= 0,\ \
\pcG_{i}(X^{\infty}_{i})\leq 0,\ \
\svG^{\infty}_{i}\geq 0,\ \
\dvG^{\infty}_{i}\geq 0,\ \
\big\langle \dvG^{\infty}_{i},\pcG_{i}(X^{\infty}_{i})\big\rangle= 0,
\notag\\[0mm]&
\pcH_{i}(X^{\infty}_{i})=
\svHp^{\infty}_{i}=
\svHn^{\infty}_{i}= 0,\ \ 
\dvHp^{\infty}_{i}\geq 0,\ \ 
\dvHn^{\infty}_{i}\geq 0.
\label{StationarityConditionsPart1}
\end{align}
 The cluster point $Z^{\infty}$ is a feasible solution of
 the primal problem~\eqref{ACPGSPrimalProblemOriginal}.

\item[(B)]
$\lim_{k\rightarrow\infty}Z^k=Z^{\infty}$,
$\lim_{k\rightarrow\infty}X^k_i=X^{\infty}_i$,
$\lim_{k\rightarrow\infty}f(X^k_i)=f(X^{\infty}_i)=f(Z^{\infty})$,
$\lim_{k\rightarrow\infty}L^{k}=Nf(Z^{\infty})$.
 The limit of $\{Z^{k}\}_{k}$ is a feasible solution of
 the primal problem~\eqref{ACPGSPrimalProblemOriginal}.

\item[(C)]
  $\lim_{k\rightarrow\infty}\proxXNormOcoef^{k}_{i,l}=0$.

\item[(D)]
$\forall i,l$,
\begin{align} 
 &
-\Big((f_{l})^T
    +\dvZp^{\infty}_{i,l}
    -\dvZn^{\infty}_{i,l}
    +\big\langle  \dvF^{\infty}_{i}, \nabla_{X_{i,l}}\pcF_{i}(X^{\infty}_{i}) \big\rangle
    +(\pcG_{i,l})^T\dvG^{\infty}_{i}
    +(\pcH_{i,l})^T \big(\dvHp^{\infty}_{i}-\dvHn^{\infty}_{i}\big)
\Big)\notag\\
\in\,& \partial_{X_{i,l}}\Indicator_{[-u_{l}, u_{l}]}(X^{\infty}_{i,l}),
\ \ \
  -\frac{1}{N}
   \sum_{i}\big(
         \dvZn^{\infty}_{i,l}-\dvZp^{\infty}_{i,l}
        \big)
  \in \partial_{Z_{l}}\Indicator_{[-u_{l},u_{l}]}(Z^{\infty}_{l}),\ \
-\dvZp^{\infty}_{i}
\in
  \partial_{\svZp_{i}}\Indicator_{[0,\usvZp_{Y_{i}}]}(\svZp^{\infty}_{i}),
\notag\\&
-\dvZn^{\infty}_{i}\in
\partial_{\svZn_{i}}\Indicator_{[0,\usvZn_{Y_{i}}]}(\svZn^{\infty}_{i}),\ \
-\dvF^{\infty}_{i}
\in
\partial_{\svF_{i}}\Indicator_{[0,\usvF_{Y_{i}}]}(\svF^{\infty}_{i}),\ \
-\dvG^{k}_{i}
\in
\partial_{\svG_{i}}\Indicator_{[0,\usvG_{Y_{i}}]}(\svG^{\infty}_{i}),
\notag\\&
-\dvHp^{\infty}_{i}\in
\partial_{\svHp_{i}}\Indicator_{[0,\usvHp_{Y_{i}}]}(\svHp^{\infty}_{i}),\ \
-\dvHn^{k}_{i}\in
\partial_{\svHn_{i}}\Indicator_{[0,\usvHn_{Y_{i}}]}(\svHn^{\infty}_{i}).
\label{StationarityConditionsPart2}
\end{align}
\end{enumerate}
\label{PDDeltaLimitsLemma}
\end{lemma}
\vspace*{-6mm}
\begin{proof}
  Of Lemma~\ref{PDDeltaLimitsLemma}.
  \begin{enumerate}
    \item[(A)]
Without loss of generality, 
we focus on $\big\{\big(\pcF_{i}(X^{k}_{i})+\svF^{k}_{i}\big)_{j^{\ast}}\big\}_{k}$.
According to Assumption~\ref{AssumptionSetII}(A,B),
it is sufficient to analyze the cases of
\begin{align}
 &
{\cal D}(\text{NA}):=
    \big\{\dvF^{k+1}_{i,j^{\ast}}: \tdvF^{k+1}_{i,j^{\ast}}\in[0,\udvF_{\mu_{i},j^{\ast}}]\big\}_{k\geq K},\ \
{\cal D}(\text{A}):=\big\{\dvF^{k+1}_{i,j^{\ast}}=\dvF^{K}_{i,j^{\ast}}:
        \tdvF^{k+1}_{i,j^{\ast}}<0\big\}_{k\geq K},
\notag\\&
{\cal D}(\text{NA,A})
:=\big\{\dvF^{k_{j}+1}_{i,j^{\ast}}: \tdvF^{k_{j}+1}_{i,j^{\ast}}\in [0,\udvF_{\mu_{i},j^{\ast}}], k_{j}\geq K\big\}_{j}
    \cup\big\{\dvF^{k'_{j}+1}_{i,j^{\ast}}: \tdvF^{k'_{j}+1}_{i,j^{\ast}}<0,
                         k'_{j}\geq K \big\}_{j}
    \notag\\&\hskip20mm
    \cup\big\{\dvF^{k''_{j}+1}_{i,j^{\ast}}: \tdvF^{k''_{j}+1}_{i,j^{\ast}}
        >\udvF_{\mu_{i},j^{\ast}}, k''_{j}\geq K \big\}_{j},
\label{NotTrappedIndvFijUpperBounds}
\end{align}
from \eqref{DualUpdateRules} and \eqref{DualUpdateRulesdvEFGBoundsj},
combined with 
\begin{align}
 &
 {\cal S}(\text{NA}):=\big\{\svF^{k+1}_{i,j^{\ast}}:
    \big(\svFcoef^{k}_{i}\svF^{k}_{i}-\rho_{i}\pcF_{i}(X^{k+1}_{i})-\dvF^{k}_{i}\big)_{j^{\ast}}
        \in[0, \usvF_{Y_{i},j^{\ast}})\big\}_{k\geq K},
\notag\\&
{\cal S}(\text{A}):=\big\{\svF^{k+1}_{i,j^{\ast}}=0:
    \big(\svFcoef^{k}_{i}\svF^{k}_{i}-\rho_{i}\pcF_{i}(X^{k+1}_{i})-\dvF^{k}_{i}\big)_{j^{\ast}}<0\big\}_{k\geq K},
\notag\\&
{\cal S}(\text{NA,A}):=\big\{\svF^{k_{j}+1}_{i,j^{\ast}}:
    \big(\svFcoef^{k_{j}}_{i}\svF^{k_{j}}_{i}-\rho_{i}\pcF_{i}(X^{k_{j}+1}_{i})-\dvF^{k_{j}}_{i}\big)_{j^{\ast}}
        \in[0, \usvF_{Y_{i},j^{\ast}}), k_{j}\geq K\big\}_{j}
        \notag\\&\hskip20mm
 \cup \big\{\svF^{k'_{j}+1}_{i,j^{\ast}}=0:
    \big(\svFcoef^{k'_{j}}_{i}\svF^{k'_{j}}_{i}-\rho_{i}\pcF_{i}(X^{k'_{j}+1}_{i})
            -\dvF^{k'_{j}}_{i}\big)_{j^{\ast}}<0, k'_{j}\geq K\big\}_{j},
\label{svFijBehaviorGrouping}
\end{align}
from \eqref{SolvingsvFk}.
\begin{enumerate}
\item
${\cal D}(\text{NA})$.
Lemma~\ref{PDSumBoundedLemma}(A,B) yield
\begin{align}
 \lim_{k\rightarrow\infty}\big(\pcF_{i}(X^{k}_{i})+\svF^{k}_{i}\big)_{j^{\ast}}=0,\ \ \
 \sum_{k\geq K} \big\vert \svF^{k+1}_{i,j^{\ast}}-\svF^{k}_{i,j^{\ast}} \big\vert^2<\infty.
 \label{dvFi(NA)}
\end{align}
\eqref{SolvingsvFk} is used to proceed further.
\begin{enumerate}
\item
${\cal S}(\text{NA})$.
Combination of \eqref{dvFi(NA)}, \eqref{SolvingsvFk},
Assumption~\ref{AssumptionSetII}(C),
and Lemma~\ref{PDAccumulationPointLemma} gives
\begin{align}
&
 \lim_{k\rightarrow\infty}\s\big(\pcF_{i}(X^{k}_{i})+\svF^{k}_{i}\big)_{j^{\ast}}
 =\lim_{k\rightarrow\infty}\dvF^{k}_{i,j^{\ast}}
 =0,\ \
\lim_{k\rightarrow\infty}\dvF^{k}_{i,j^{\ast}}\big(\pcF_{i}(X^{k}_{i})\big)_{j^{\ast}}=0,
\notag\\&
\big(\pcF_{i}(X^{\infty}_{i})\big)_{j^{\ast}}\leq 0,\ \
\svF^{\infty}_{i,j^{\ast}}\geq 0,\ \
\dvF^{\infty}_{i,j^{\ast}}=0,\ \  
\dvF^{\infty}_{i,j^{\ast}}\big(\pcF_{i}(X^{\infty}_{i})\big)_{j^{\ast}}= 0.
\label{dvFi(NA).(NA)Limits}
\end{align}

\item
${\cal S}(\text{NA,A})$.
The component sequence $\{\svF^{k}_{i,j^{\ast}}\}_{k}$ has a zero cluster point.
It then follows from (\ref{dvFi(NA)}$)_2$ 
and Assumption~\ref{AssumptionSetII}(D) via contradiction that
\begin{align}
\lim_{k\rightarrow\infty}\svF^{k}_{i,j^{\ast}}=0.
\label{dvFi(NA).(NA,A)Relations01}
\end{align}
Its combination with (\ref{dvFi(NA)}$)_1$
and Lemma~\ref{PDAccumulationPointLemma} results in
\begin{align}
&
 \lim_{k\rightarrow\infty}\big(\pcF_{i}(X^{k}_{i})\big)_{j^{\ast}}
=\lim_{k\rightarrow\infty}\svF^{k}_{i,j^{\ast}}=0,\ \
\lim_{k\rightarrow\infty}\dvF^{k}_{i,j^{\ast}}\big(\pcF_{i}(X^{k}_{i})\big)_{j^{\ast}}=0,
\notag\\&
\big(\pcF_{i}(X^{\infty}_{i})\big)_{j^{\ast}}=0,\ \
\svF^{\infty}_{i,j^{\ast}}=0,\ \ 
\dvF^{\infty}_{i,j^{\ast}}\geq 0,\ \ 
\dvF^{\infty}_{i,j^{\ast}}\big(\pcF_{i}(X^{\infty}_{i})\big)_{j^{\ast}}= 0.
\label{dvFi(NA).(NA,A)Limits}
\end{align}

\item
${\cal S}(\text{A})$.
\eqref{SolvingsvFk} results in
\begin{align*}
\svF^{k+1}_{i,j^{\ast}}=0,\ \
    \big(\svFcoef^{k}_{i}\svF^{k}_{i}-\rho_{i}\pcF_{i}(X^{k+1}_{i})-\dvF^{k}_{i}\big)_{j^{\ast}}<0\ \
      \forall k\geq K.
\end{align*}
Then,
\begin{align}
&
 \lim_{k\rightarrow\infty}\s\big(\pcF_{i}(X^{k}_{i})\big)_{j^{\ast}}
 =\lim_{k\rightarrow\infty}\svF^{k}_{i,j^{\ast}}=0,\ \
\lim_{k\rightarrow\infty}\dvF^{k}_{i,j^{\ast}}\big(\pcF_{i}(X^{k}_{i})\big)_{j^{\ast}}=0,
\notag\\&
\big(\pcF_{i}(X^{\infty}_{i})\big)_{j^{\ast}}=0,\ \
\svF^{\infty}_{i,j^{\ast}}=0,\ \
\dvF^{\infty}_{i,j^{\ast}}\geq 0,\ \
\dvF^{\infty}_{i,j^{\ast}}\big(\pcF_{i}(X^{\infty}_{i})\big)_{j^{\ast}}= 0.
\label{dvFi(NA).(A)Limits}
\end{align}
\end{enumerate}

\item
 ${\cal D}(\text{NA,A})$.
Lemma~\ref{PDSumBoundedLemma}(A,B) yield
\begin{align}
  &
  \sum_{k_j\geq K:\,\dvFcoef^{k_{j}+1}_{i,j^{\ast}}=\dvFcoef_{i}}\sss
      \big\vert \big(\pcF_{i}(X^{k_{j}+1}_{i})+\svF^{k_{j}+1}_{i}\big)_{j^{\ast}}\big\vert^2
  <\infty,\ \
\sum_{k=0}^{\infty}\s \big\vert \svF^{k+1}_{i,j^{\ast}}-\svF^{k}_{i,j^{\ast}} \big\vert^2<\infty,\ \
\sum_{k=0}^{\infty}\s \big\Vert X^{k+1}_{i}-X^{k}_{i} \big\Vert^2<\infty.
\label{dvFi(NA,0).Relaions01}
\end{align}
Therefore,
$\big\{\big(\pcF_{i}(X^{k+1}_{i})+\svF^{k+1}_{i}\big)_{j^{\ast}}\big\}_{k}$
has a zero cluster point.
Further, the special structures of the convex functions $\pcF_{i}$
(restricted to $[-u,u]$),
(\ref{dvFi(NA,0).Relaions01}$)_{2,3}$,
and Assumption~\ref{AssumptionSetII}(D) result in
\begin{align}
 \lim_{k\rightarrow\infty}\big(\pcF_{i}(X^{k}_{i})+\svF^{k}_{i}\big)_{j^{\ast}}
=0,
\label{dvFi(NA,0).Relaions02}
\end{align}
via contradiction.
\begin{enumerate}
\item
${\cal S}(\text{NA})$.
Combination of \eqref{SolvingsvFk}, (\ref{dvFi(NA,0).Relaions01}$)_2$,
and \eqref{dvFi(NA,0).Relaions02} gives
\begin{align}
&
\lim_{k\rightarrow\infty}\big(\pcF_{i}(X^{k}_{i})+\svF^{k}_{i}\big)_{j^{\ast}}
=\lim_{k\rightarrow\infty}\dvF^{k}_{i,j^{\ast}}=0,\ \
\lim_{k\rightarrow\infty}\dvF^{k}_{i,j^{\ast}}\big(\pcF_{i}(X^{k}_{i})\big)_{j^{\ast}}=0,
\notag\\&
\big(\pcF_{i}(X^{\infty}_{i})\big)_{j^{\ast}}\leq 0,\ \
\svF^{\infty}_{i,j^{\ast}}\geq 0,\ \ 
\dvF^{\infty}_{i,j^{\ast}}=0,\ \ 
\dvF^{\infty}_{i,j^{\ast}}\big(\pcF_{i}(X^{\infty}_{i})\big)_{j^{\ast}}= 0. 
\label{dvFi(NA,0).(NA)Limits}
\end{align}

\item
${\cal S}(\text{NA,A})$.
\eqref{dvFi(NA).(NA,A)Relations01} holds here.
Then, \eqref{dvFi(NA,0).Relaions02} yields
\begin{align}
&
 \lim_{k\rightarrow\infty}\big(\pcF_{i}(X^{k}_{i})\big)_{j^{\ast}}
=\lim_{k\rightarrow\infty}\svF^{k}_{i,j^{\ast}}=0,\ \
\lim_{k\rightarrow\infty} \dvF^{k}_{i,j^{\ast}}\big(\pcF_{i}(X^{k}_{i})\big)_{j^{\ast}}=0,
\notag\\&
\big(\pcF_{i}(X^{\infty}_{i})\big)_{j^{\ast}}=0,\ \
\svF^{\infty}_{i,j^{\ast}}=0,\ \ 
\dvF^{\infty}_{i,j^{\ast}}\geq 0,\ \ 
\dvF^{\infty}_{i,j^{\ast}}\big(\pcF_{i}(X^{\infty}_{i})\big)_{j^{\ast}}= 0.
\label{dvFi(NA,0).(NA,A)Limits}
\end{align}

\item
${\cal S}(\text{A})$.
 \eqref{SolvingsvFk} and \eqref{dvFi(NA,0).Relaions02} give
\begin{align}
&
\lim_{k\rightarrow\infty} \big(\pcF_{i}(X^{k}_{i})\big)_{j^{\ast}}
=\lim_{k\rightarrow\infty} \svF^{k}_{i,j^{\ast}}=0,\ \
\lim_{k\rightarrow\infty} \dvF^{k}_{i,j^{\ast}}\big(\pcF_{i}(X^{k}_{i})\big)_{j^{\ast}}=0,
\notag\\&
\big(\pcF_{i}(X^{\infty}_{i})\big)_{j^{\ast}}=0,\ \
\svF^{\infty}_{i,j^{\ast}}=0,\ \
\dvF^{\infty}_{i,j^{\ast}}\geq 0,\ \
\dvF^{\infty}_{i,j^{\ast}}\big(\pcF_{i}(X^{\infty}_{i})\big)_{j^{\ast}}= 0.
\label{dvFi(NA,0).(A)Limits}
\end{align}
\end{enumerate}

\item
${\cal D}(\text{A})$. The dual models~\eqref{DualUpdateRules}
and \eqref{DualUpdateRulesdvEFGBoundsj} reduce to
\begin{align}
 &
 \dvF^{k+1}_{i,j^{\ast}}=\dvF^{K}_{i,j^{\ast}},\ \
 \dvF^{K}_{i,j^{\ast}}/\dvFcoef_{i}
<\big(\pcF_{i}(X^{k+1}_{i})+\svF^{k+1}_{i}\big)_{j^{\ast}}\ \
        \forall k\geq K.
\label{dvFi(A).Relations01}
\end{align}
\begin{enumerate}
\item
${\cal S}(\text{NA})$.
  \eqref{SolvingsvFk},
  $\sum_{k\geq K} \big\vert \svF^{k+1}_{i,j^{\ast}}-\svF^{k}_{i,j^{\ast}} \big\vert^2<\infty$ of Lemma~\ref{PDSumBoundedLemma}(B),
  and \eqref{dvFi(A).Relations01} give
\begin{align}
 \lim_{k\rightarrow\infty}\big(\pcF_{i}(X^{k+1}_{i})+\svF^{k+1}_{i}\big)_{j^{\ast}}
        = -\dvF^{K}_{i,j^{\ast}}/\rho_{i},\ \
\dvF^{K}_{i,j^{\ast}}/\dvFcoef_{i} \leq -\dvF^{K}_{i,j^{\ast}}/\rho_{i},\ \
\dvF^{K}_{i,j^{\ast}}=0.
\end{align}
That is,
\begin{align}
&
\lim_{k\rightarrow\infty}\big(\pcF_{i}(X^{k}_{i})+\svF^{k}_{i}\big)_{j^{\ast}}=
\lim_{k\rightarrow\infty} \dvF^{k}_{i,j^{\ast}}=0,\ \
\lim_{k\rightarrow\infty}\dvF^{k}_{i,j^{\ast}}\big(\pcF_{i}(X^{k}_{i})\big)_{j^{\ast}}=0,
\notag\\&
\big(\pcF_{i}(X^{\infty}_{i})\big)_{j^{\ast}}\leq 0,\ \
\svF^{\infty}_{i,j^{\ast}}\geq 0,\ \
\dvF^{\infty}_{i,j^{\ast}}=0,\ \
\dvF^{\infty}_{i,j^{\ast}}\big(\pcF_{i}(X^{\infty}_{i})\big)_{j^{\ast}}=0.
\label{dvFi(A).(NA)Limits}
\end{align}

\item
${\cal S}(\text{NA,A})$.
The subsequence (NA), Assumption~\ref{AssumptionSetII}(C),
$\sum_{k\geq K} \big\vert \svF^{k+1}_{i,j^{\ast}}-\svF^{k}_{i,j^{\ast}} \big\vert^2<\infty$
of Lemma~\ref{PDSumBoundedLemma}(B),
and \eqref{SolvingsvFk}
imply that $\big\{(\rho_{i}\pcF_{i}(X^{k+1}_{i})+\dvF^{k}_{i})_{j^{\ast}}
        +\rho_{i}\svF^{k}_{i,j^{\ast}}\big\}_{k\geq K}$ has a zero cluster point.
\eqref{dvFi(NA).(NA,A)Relations01} holds here too.
It then follows that
$\big\{(\rho_{i}\pcF_{i}(X^{k+1}_{i})+\dvF^{k}_{i})_{j^{\ast}}\big\}_{k\geq K}$
has a zero cluster point.
Next,
$\sum_{k\geq K} \big\Vert X^{k+1}_{i}-X^{k}_{i} \big\Vert^2<\infty$
of Lemma~\ref{PDSumBoundedLemma}(B),
the special structures of $F_{i}$
(restricted to $[-u,u]$),
(\ref{dvFi(A).Relations01}$)_1$,
and Assumption~\ref{AssumptionSetII}(D) result in
\begin{align}
 \lim_{k\rightarrow\infty}
  \big(\rho_{i}\pcF_{i}(X^{k+1}_{i})+\dvF^{k}_{i}\big)_{j^{\ast}}
=\lim_{k\rightarrow\infty}
  \big(\rho_{i}\pcF_{i}(X^{k+1}_{i})+\dvF^{K}_{i}\big)_{j^{\ast}}
=\lim_{k\rightarrow\infty}\big(\pcF_{i}(X^{k+1}_{i})\big)_{j^{\ast}}
    +\dvF^{K}_{i,j^{\ast}}/\rho_{i}=0.
\label{dvFi(A).(NA,A)Relations02}
\end{align}
Their combination with (\ref{dvFi(A).Relations01}$)_2$
and \eqref{dvFi(NA).(NA,A)Relations01}
yields $\dvF^{K}_{i,j^{\ast}}=0$.
Finally, we have
\begin{align}
&
\lim_{k\rightarrow\infty}\big(\pcF_{i}(X^{k}_{i})\big)_{j^{\ast}}=
\lim_{k\rightarrow\infty}\svF^{k}_{i,j^{\ast}}=
\lim_{k\rightarrow\infty} \dvF^{k}_{i,j^{\ast}}=0,\ \
\lim_{k\rightarrow\infty}\dvF^{k}_{i,j^{\ast}}\big(\pcF_{i}(X^{k}_{i})\big)_{j^{\ast}}=0,
\notag\\&
\big(\pcF_{i}(X^{\infty}_{i})\big)_{j^{\ast}}=
\svF^{\infty}_{i,j^{\ast}}=
\dvF^{\infty}_{i,j^{\ast}}= 0,\ \
\dvF^{\infty}_{i,j^{\ast}}\big(\pcF_{i}(X^{\infty}_{i})\big)_{j^{\ast}}= 0.
\label{dvFi(A).(NA,A)Limits}
\end{align}

\end{enumerate}
\end{enumerate}

\item[(B)]
Combination of 
$\lim_{j\rightarrow\infty}Z^{k_j}=Z^{\infty}$ 
(from Lemma~\ref{PDAccumulationPointLemma}),
$\sum_{k=0}^{\infty}\big\Vert Z^{k+1}_{}-Z^{k}_{}\big\Vert^2<\infty$
(from Lemma~\ref{PDSumBoundedLemma}(B)), 
and Assumption~\ref{AssumptionSetII}(D)
yields $\lim_{k\rightarrow\infty}Z^k=Z^{\infty}$.
Then, the convergence of $\{X^{k}_{i}\}_{k}$
comes from Lemma~\ref{PDDeltaLimitsLemma}(A).

\item[(C)]
The limit,
$\lim_{k\rightarrow\infty}\proxXNormOcoef^{k}_{i,l}=0$
in Lemma~\ref{PDDeltaLimitsLemma}(C) comes from
\eqref{proxXNormOcoefkilDefn}, \eqref{Ukil}, Lemma~\ref{PDDeltaLimitsLemma}(A),
and $\nabla_{X_{i,l}}\nabla_{X_{i,l}}\pcF_{i}(X^{k+1,k+1}_{i,l})$ bounded.
 
\item[(D)]
The stationarity conditions in Lemma~\ref{PDDeltaLimitsLemma}(D)
follow from \eqref{SolvingXil} through \eqref{SolvingsvHnk},
Lemma~\ref{PDDeltaLimitsLemma}(A),
and Lemma~\ref{PDAccumulationPointLemma}.
\end{enumerate}
\label{ProofOfPDDeltaLimitsLemma}
\end{proof}

We can also construct the first-order characteristics of the convex functions
involved in
\eqref{SolvingXil} through \eqref{SolvingsvHnk}
under $X_{i,l}=Z^{\ast}_{l}, Z_{l}=Z^{\ast}_{l}, \svZp_{i}=\svZp^{\ast}_{i}=0,
\svZn_{i}=\svZn^{\ast}_{i}=0,
 \svF_{i}=\svF^{\ast}_{i}\equiv -\pcF_{i}(Z^{\ast}),
 \svG_{i}=\svG^{\ast}_{i}\equiv -\pcG_{i}(Z^{\ast}),
 \svHp_{i}=\svHp^{\ast}_{i}=0$, and $\svHn_{i}=\svHn^{\ast}_{i}=0$.
 Summation of the characteristics results in the following,
\begin{align} 
&
\sum_{i}\Big[f(Z^{\ast})
      -\big\langle \dvF^{k}_{i}, 
          \pcF_{i}(X^{\ast}_{i})-\pcF_{i}(X^{k+1}_{i})
         -(X^{\ast}_{i}-X^{k+1}_{i})^{T}\,\nabla_{X_{i}}\pcF_{i}(X^{k+1}_{i})
            \big\rangle\Big]
\notag\\[-2mm]
\geq\,&
\sum_{i}\Big[
     f(X^{k+1}_{i})
    +\big\langle \dvZp^{k}_{i},  X^{k+1}_{i}-Z^{k+1}_{}+\svZp^{k+1}_{i} \big\rangle
    +\big\langle \dvZn^{k}_{i},  Z^{k+1}_{}-X^{k+1}_{i}+\svZn^{k+1}_{i} \big\rangle
\notag\\[-3mm]&\hskip8mm
+\big\langle \dvF^{k}_{i}, \svF^{k+1}_{i}+\pcF_{i}(X^{k+1}_{i})\big\rangle
+\sum_{l}\big\langle  \dvF^{k}_{i}, 
          (X^{\ast}_{i,l}-X^{k+1}_{i,l})^{T}\,
            \big(\nabla_{X_{i,l}}\pcF_{i}(X^{k+1}_{i})
                        -\nabla_{X_{i,l}}\pcF_{i}(X^{k+1,k+1}_{i,l})\big)
            \big\rangle
\notag\\[-2mm]&\hskip8mm
+\big\langle \dvG^{k}_{i}, \pcG_{i}(X^{k+1}_{i})+\svG^{k+1}_{i} \big\rangle 
+\big\langle \dvHp^{k}_{i}, \pcH_{i}(X^{k+1}_{i})+\svHp^{k+1}_{i}\big\rangle
+\big\langle \dvHn^{k}_{i}, -\pcH_{i}(X^{k+1}_{i})+\svHn^{k+1}_{i} \big\rangle
\notag\\&\hskip8mm
+\rho_{i}\big\langle X^{k+1}_{i}-Z^{k+1}_{}+\svZp^{k+1}_{i},
        \svZp^{k+1}_{i} \big\rangle
+\rho_{i}\big\langle Z^{k+1}_{}-X^{k+1}_{i}+\svZn^{k+1}_{i},
        \svZn^{k+1}_{i} \big\rangle
\notag\\&\hskip8mm
+\rho_{i}\big\langle \pcF_{i}(X^{k+1}_{i})+\svF^{k+1}_{i},
        \svF^{k+1}_{i}+\pcF_{i}(Z^{\ast}) \big\rangle
+\rho_{i}\big\langle \pcG_{i}(X^{k+1}_{i})+\svG^{k+1}_{i}, 
        \svG^{k+1}_{i}+\pcG_{i}(Z^{\ast}) \big\rangle
\notag\\&\hskip8mm
+\rho_{i}\big\langle \pcH_{i}(X^{k+1}_{i})+\svHp^{k+1}_{i},
        \svHp^{k+1}_{i}\big\rangle
+\rho_{i}\big\langle -\pcH_{i}(X^{k+1}_{i})+\svHn^{k+1}_{i},
        -\svHn^{k+1}_{i} \big\rangle
\notag\\&\hskip8mm
+\rho_{i}\big\langle (X^{k+1}_{i}-Z^{k}_{}-\svZn^{k}_{i})
                        +(X^{k+1}_{i}-Z^{k}_{}+\svZp^{k}_{i}),
                    X^{k+1}_{i}-Z^{\ast} \big\rangle
\notag\\&\hskip8mm
+\big\langle \rho_{i}\big(
                           Z^{k+1}_{}-X^{k+1}_{i}-\svZp^{k}_{i}
                          +Z^{k+1}_{}-X^{k+1}_{i}+\svZn^{k}_{i}
                                    \big)
                    +\tau^{k}(Z^{k+1}_{}-Z^{k}_{})
                    ,
                    Z^{k+1}_{}-Z^{\ast}_{}
        \big\rangle
\notag\\&\hskip8mm
+\rho_{i}\sum_{l}\big\langle
        \pcF_{i}(X^{k+1,k+1}_{i,l})-\pcF_{i}(X^{k}_{i})+\pcF_{i}(X^{k}_{i})+\svF^{k}_{i},
                (X^{k+1}_{i,l}-X^{\ast}_{i,l})^{T}\,
                    \nabla_{X_{i,l}}\pcF_{i}(X^{k+1,k+1}_{i,l}) \big\rangle
\notag\\[-2mm]&\hskip8mm
+\rho_{i}\big\langle \pcG_{i}(X^{k+1}_{i})+\svG^{k}_{i},
                \pcG_{i}(X^{\ast}_{i})-\pcG_{i}(X^{k+1}_{i}) \big\rangle
+\rho_{i}\sum_{l}
        \big\langle \pcG_{i}(X^{k+1,k+1}_{i,l})-\pcG_{i}(X^{k+1}_{i}),
                \pcG_{i,l}(X^{k+1}_{i,l}-Z^{\ast}_{l}) \big\rangle
\notag\\[-1mm]&\hskip8mm
+\rho_{i}\big\langle \pcH_{i}(X^{k+1}_{i})+\svHp^{k}_{i}+\pcH_{i}(X^{k+1}_{i})-\svHn^{k}_{i},
                \pcH_{i}(X^{k+1}_{i}) \big\rangle
\notag\\[-1mm]&\hskip8mm
+2\rho_{i}\sum_{l}
        \big\langle \pcH_{i}(X^{k+1,k+1}_{i,l})-\pcH_{i}(X^{k+1}_{i}),
                \pcH_{i,l}(X^{k+1}_{i,l}-Z^{\ast}_{l}) \big\rangle
\notag\\[-1mm]&\hskip8mm
+\svZpcoef^{k}_{i}\big\langle \svZp^{k+1}_{i}-\svZp^{k}_{i},
        \svZp^{k+1}_{i} \big\rangle
+\svZncoef^{k}_{i}\big\langle \svZn^{k+1}_{i}-\svZn^{k}_{i}, 
        \svZn^{k+1}_{i} \big\rangle
\notag\\&\hskip8mm
+\svFcoef^{k}_{i}\big\langle \svF^{k+1}_{i}-\svF^{k}_{i}, 
        \svF^{k+1}_{i}+\pcF_{i}(Z^{\ast}) \big\rangle
+\svGcoef^{k}_{i}\big\langle \svG^{k+1}_{i}-\svG^{k}_{i},
        \svG^{k+1}_{i}+\pcG_{i}(Z^{\ast}) \big\rangle
\notag\\&\hskip8mm
+\svHpcoef^{k}_{i}\big\langle \svHp^{k+1}_{i}-\svHp^{k}_{i},
        \svHp^{k+1}_{i}\big\rangle
+\svHncoef^{k}_{i}\big\langle \svHn^{k+1}_{i}-\svHn^{k}_{i},
        \svHn^{k+1}_{i} \big\rangle
\notag\\&\hskip8mm
+\sum_{l}\proxXNormOcoef^{k}_{i,l}
        \big(\Vert X^{k+1}_{i,l}-X^{k}_{i,l} \Vert_{1}
                    -\Vert X^{\ast}_{i,l}-X^{k}_{i,l} \Vert_{1}\big)
+\sum_{l}\proxXNormTcoef^{k}_{i,l}
        \big\langle X^{k+1}_{i,l}-X^{k}_{i,l}, X^{k+1}_{i,l}-Z^{\ast}_{l} \big\rangle
\Big]\ \ \forall k.
\label{GlobalInequalityast}
\end{align}
Next, with the help of Lemma~\ref{PDDeltaLimitsLemma}(A)(B)(C),
the continuity of $\{\pcF_{i},\pcG_{i},\pcH_{i},\nabla_{X_{i,l}}\pcF_{i}\}$,
and $\pcF_{i}(X^{\ast}_{i})-\pcF_{i}(X^{k+1}_{i})
         -(X^{\ast}_{i}-X^{k+1}_{i})^{T}\,\nabla_{X_{i}}\pcF_{i}(X^{k+1}_{i}) \geq 0$,
we operate $\overline{\lim}_{k\rightarrow\infty}$ on
\eqref{GlobalInequalityast} to obtain  $f(Z^{\ast})\geq f(Z^{\infty})$,
which is stated in 
\begin{lemma}
  $f(Z^{\ast})\geq f(Z^{\infty})$ and
$Z^{\infty}$ is an optimal solution 
to the primal problem~\eqref{ACPGSPrimalProblemOriginal}.
\label{OptimalSolLemma} 
\end{lemma}

It is noticed that in the above argument,
\eqref{GlobalInequalityast} yields
$\big\langle \dvF^{\infty}_{i},
          \pcF_{i}(X^{\ast}_{i})-\pcF_{i}(X^{\infty}_{i})
         -(X^{\ast}_{i}-X^{\infty}_{i})^{T}\,\nabla_{X_{i}}\pcF_{i}(X^{\infty}_{i})
            \big\rangle= 0$;
the $j$-th component $\dvF^{\infty}_{i,j}>0$ leads to
\begin{align}
&
(\pcF_{i}(Z^{}))_j
=(\pcF_{i}(Z^{\infty}))_j
+(Z^{}-Z^{\infty})^{T}\,\nabla_{Z}(\pcF_{i}(Z^{\infty}))_j
\ \ \forall Z \in \{ tZ^{\ast}+(1-t)Z^{\infty}, t\in[0,1]\},
\label{OptimalSoluLinkedviaTaylorSExpan}
\end{align}
each $Z$ in \eqref{OptimalSoluLinkedviaTaylorSExpan} is optimal,
since $f$ is linear.
\eqref{OptimalSoluLinkedviaTaylorSExpan} is
a first-order Taylor series expansion
of $(\pcF_{i}(Z^{}))_j$ at $Z^{\infty}$
and the quadratic part is absent,
which implies that some components of $Z^{\infty}$ are the same as
that of $Z^{\ast}$, depending on the structures of
$(\pcF_{i}(Z^{}))_j$ for all $i$ and $j$.
  This issue is to be examined further.

Summarizing the above lemmas, we have
\begin{theorem}
Suppose that Assumption~\ref{AssumptionSetII} holds.
 Then,
\begin{enumerate}
\item[(A)] 
The sequences, $\{X^{k}_{i}$, $Z^{k}$, $\svZp^{k}_{i}$,
  $\svZn^{k}_{i}$, $\svF^{k}_{i}$, $\svG^{k}_{i}$,
  $\svHp^{k}_{i}$, $\svHn^{k}_{i}\}_{k}$
  converge with the limits, $X^{\infty}_{i}=Z^{\infty}$,
  $\svZp^{\infty}_{i}=\svZn^{\infty}_{i}=0$,
  $\svF^{\infty}_{i}\geq 0$, $\svG^{\infty}_{i}\geq 0$,
  and $\svHp^{\infty}_{i}=\svHn^{\infty}_{i}=0$,
  satisfying
$\pcF_{i}(Z^{\infty})\leq 0$,
$\pcG_{i}(Z^{\infty})\leq 0$,
$\pcH_{i}(Z^{\infty})=0$ $\forall i$.
\item[(B)]
$\dvZp^{\infty}_{i}\geq 0$,
$\dvZn^{\infty}_{i}\geq 0$,
$\dvF^{\infty}_{i}\geq 0$,
$\dvG^{\infty}_{i}\geq 0$,
$\dvHp^{\infty}_{i}\geq 0$,
$\dvHn^{\infty}_{i}\geq 0$,
$\big\langle \dvF^{\infty}_{i},\pcF_{i}(X^{\infty}_{i})\big\rangle= 0$,
$\big\langle \dvG^{\infty}_{i},\pcG_{i}(X^{\infty}_{i})\big\rangle= 0$,
and
\begin{align}
  &
-\Big(\s(f_{l})^T
    +\dvZp^{\infty}_{i,l}
    -\dvZn^{\infty}_{i,l}
    +(\pcG_{i,l})^T\s\dvG^{\infty}_{i}
    +(\pcH_{i,l})^T\s\big(\dvHp^{\infty}_{i}-\dvHn^{\infty}_{i}\big)
    +\big\langle  \dvF^{\infty}_{i}, \nabla_{X_{i,l}}\pcF_{i}(X^{\infty}_{i}) \big\rangle\s
\Big)
  \in \partial_{X_{i,l,}}\s\Indicator_{[-u_{l}, u_{l}]}(X^{\infty}_{i,l}),
\notag\\&
  -\frac{1}{N}
   \sum_{i}\big(
         \dvZn^{\infty}_{i,l}-\dvZp^{\infty}_{i,l}
        \big)
  \in \partial_{Z_{l}}\Indicator_{[-u_{l},u_{l}]}(Z^{\infty}_{l}),\ \
-\dvZp^{\infty}_{i}
\in
  \partial_{\svZp_{i}}\Indicator_{[0,\usvZp_{Y_{i}}]}(\svZp^{\infty}_{i}),
\notag\\&
-\dvZn^{\infty}_{i}\in
\partial_{\svZn_{i}}\Indicator_{[0,\usvZn_{Y_{i}}]}(\svZn^{\infty}_{i}),\ \
-\dvF^{\infty}_{i}
\in
\partial_{\svF_{i}}\Indicator_{[0,\usvF_{Y_{i}}]}(\svF^{\infty}_{i}),\ \
-\dvG^{k}_{i}
\in
\partial_{\svG_{i}}\Indicator_{[0,\usvG_{Y_{i}}]}(\svG^{\infty}_{i}),
\notag\\&
-\dvHp^{\infty}_{i}\in
\partial_{\svHp_{i}}\Indicator_{[0,\usvHp_{Y_{i}}]}(\svHp^{\infty}_{i}),\ \
-\dvHn^{k}_{i}\in
\partial_{\svHn_{i}}\Indicator_{[0,\usvHn_{Y_{i}}]}(\svHn^{\infty}_{i}).
\label{StationarityConditionsThm1}
\end{align}
\item[(C)] 
$\{L^{k}\}_{k}$ converges and
$\lim_{k\rightarrow\infty}  L^{k}=N f(Z^{\infty})$.
\item[(D)] 
  $\lim_{k\rightarrow\infty} f(X^{k}_{i})
      =\lim_{k\rightarrow\infty} f(Z^{k})=f(Z^{\infty})$, and 
 $Z^{\infty}$ is an optimal solution of 
the primal problem~\eqref{ACPGSPrimalProblemOriginal}.
\end{enumerate}
\label{ConvergenceTheorem}
\end{theorem}

As done in \cite{Tao2025v7LP},
we estimate the rate of convergence of the algorithm composed of 
\eqref{ACPGSProximalXilA},
\eqref{ACPGSProximalZl},
\eqref{ACPGSProximalsvZpiA} through \eqref{ACPGSProximalsvHniA},
\eqref{DualUpdateRules}, \eqref{DualUpdateRulesdvEFGBoundsj},
\eqref{proxXNormOcoefkilDefn}, and \eqref{ProximalControlParameter2sigmaA}
through \eqref{ProximalControlParameternHgammaA},
as follows.
It is observed that $X^{\infty}_{i}=Z^{\infty}$ is a direct consequence
of the bounded sums of the squared terms in Lemma~\ref{PDSumBoundedLemma}(A)(B)
(see Proof~\ref{ProofOfPDDeltaLimitsLemma}(A)(B) also),
 therefore, these bounded sums may be used to estimate
 roughly the rate of convergence of the algorithm.
To this end, recall the well-known result of
$\sum_{k=1}^{\infty}k^{-q}<\infty$ if and only if $q>1$,
we then infer that there exists $\Lambda>0$ such that
 $\Vert X^{k}_{i}-Z^{k}_{}+\svZp^{k}_{i} \Vert^2<\Lambda^2/k$
 for almost all great $k\in\mathbb{N}$ (or on average).
Consequently, $O(1/k^{1/2})$ is a rough estimate
of the rate of convergence of the algorithm,
stated in Theorem~\ref{ConvgRateTheorem}.
\begin{theorem}
(\textbf{$O(1/k^{1/2})$ rate of convergence}).
Suppose that Assumption~\ref{AssumptionSetII} holds.
Then, $O(1/k^{1/2})$ is a rough estimate of the rate of convergence of
the algorithm composed of 
\eqref{ACPGSProximalXilA},
\eqref{ACPGSProximalZl},
\eqref{ACPGSProximalsvZpiA} through \eqref{ACPGSProximalsvHniA},
\eqref{DualUpdateRules}, \eqref{DualUpdateRulesdvEFGBoundsj},
\eqref{proxXNormOcoefkilDefn}
and \eqref{ProximalControlParameter2sigmaA}
through \eqref{ProximalControlParameternHgammaA}.
\label{ConvgRateTheorem} 
\end{theorem}

\section{Parameter Estimation, Initialization, and Feasibility Conditions}
\label{sec:Initialization}

Assumption~\ref{AssumptionSetII} provides a ground to establish
Theorems~\ref{ConvergenceTheorem} and \ref{ConvgRateTheorem}.
How to realize these conditions supposed is discussed in some
more details in this section.
To this end, we obtain from
Lemma~\ref{GlobalInequalitydvSum2KLemma} and Theorem~\ref{ConvergenceTheorem},
\begin{align} 
L^{0}
\geq\,&
 N f(Z^{\infty})
+\sum_{i}\sum_{j}\sum_{k=1}^{\infty}D^{k}_{i,j}
+\sum_{i}\sum_{k=0}^{\infty}P^{k}_{i}
+\sum_{i}\sum_{k=0}^{\infty}\sum_{l}
    \Big( \proxXNormOcoef^{k}_{i,l}\,\Vert X^{k}_{i,l}-X^{k+1}_{i,l}\Vert_{1}
         +U^{k}_{i,l}\Big),
\label{GlobalInequalitydvSum2infty}
\end{align}
where
\begin{align}
L^{0}_{i}=\,&
     f(X^{0}_{i})
    +\big\langle \dvZp^{0}_{i}, X^{0}_{i}-Z^{0}_{}+\svZp^{0}_{i} \big\rangle
    +\big\langle \dvZn^{0}_{i}, Z^{0}_{}-X^{0}_{i}+\svZn^{0}_{i} \big\rangle
    +\big\langle \dvF^{0}_{i}, \pcF_{i}(X^{0}_{i})+\svF^{0}_{i} \big\rangle
    +\big\langle \dvG^{0}_{i}, \pcG_{i}(X^{0}_{i})+\svG^{0}_{i} \big\rangle
    \notag\\&
    +\big\langle \dvHp^{0}_{i}, \pcH_{i}(X^{0}_{i})+\svHp^{0}_{i}\big\rangle
    +\big\langle \dvHn^{0}_{i}, -\pcH_{i}(X^{0}_{i})+\svHn^{0}_{i} \big\rangle
    +\frac{\rho_{i}}{2}\Big(\big\Vert X^{0}_{i}-Z^{0}_{}+\svZp^{0}_{i} \big\Vert^2 
            +\big\Vert Z^{0}_{}-X^{0}_{i}+\svZn^{0}_{i} \big\Vert^2 
            \notag\\[-3pt]&\hskip5mm
            +\big\Vert \pcF_{i}(X^{0}_{i})+\svF^{0}_{i} \big\Vert^2 
            +\big\Vert \pcG_{i}(X^{0}_{i})+\svG^{0}_{i} \big\Vert^2 
            +\big\Vert \pcH_{i}(X^{0}_{i})+\svHp^{0}_{i} \big\Vert^2 
            +\big\Vert -\pcH_{i}(X^{0}_{i})+\svHn^{0}_{i} \big\Vert^2 
    \Big),\ L^{0}=\sum_{i}L^{0}_{i},
\label{LkDefnki0}
\end{align}
\vspace{-5mm}
\begin{align}
D^{k}_{i,j}:=\,&
 \dvZpcoef^{k}_{i,j}\big\vert \big(X^{k}_{i}-Z^{k}_{}+\svZp^{k}_{i}\big)_{j}\big\vert^2
+\dvZncoef^{k}_{i,j}\big\vert \big(Z^{k}_{}-X^{k}_{i}+\svZn^{k}_{i}\big)_{j}\big\vert^2
+\dvFcoef^{k}_{i,j}\big\vert \big(\pcF_{i}(X^{k}_{i})+\svF^{k}_{i}\big)_{j}\big\vert^2
\notag\\[-0mm]&
+\dvGcoef^{k}_{i,j}\big\vert \big(\pcG_{i}(X^{k}_{i})+\svG^{k}_{i}\big)_{j}\big\vert^2
+\dvHpcoef^{k}_{i,j}\big\vert \big(\pcH_{i}(X^{k}_{i})+\svHp^{k}_{i}\big)_{j}\big\vert^2
+\dvHncoef^{k}_{i,j}\big\vert \big(-\pcH_{i}(X^{k}_{i})+\svHn^{k}_{i}\big)_{j}\big\vert^2,
\label{GlobalInequalityDkijDefn}
\end{align}
\vspace{-5mm}
\begin{align}
 P^{k}_{i}
 =\,&
\sum_{l}(\proxXNormTcoef^{k}_{i,l}+\rho_{i})\big\Vert X^{k+1}_{i,l}-X^{k}_{i,l} \big\Vert^2
+(\tau^{k}+\rho_{i})\big\Vert Z^{k+1}_{}-Z^{k}_{}\big\Vert^2
+\big(\svZpcoef^{k}_{i}+\frac{\rho_{i}}{2}\big)\big\Vert \svZp^{k+1}_{i}-\svZp^{k}_{i}\big\Vert^2
\notag\\[-4pt]&
+\big(\svZncoef^{k}_{i}+\frac{\rho_{i}}{2}\big)\big\Vert\svZn^{k+1}_{i}-\svZn^{k}_{i}\Vert^2
+\big(\svFcoef^{k}_{i}+\frac{\rho_{i}}{2}\big)
                \big\Vert \svF^{k+1}_{i}-\svF^{k}_{i} \big\Vert^2
+\big(\svGcoef^{k}_{i}+\frac{\rho_{i}}{2}\big)\big\Vert \svG^{k+1}_{i}-\svG^{k}_{i} \big\Vert^2
\notag\\&
+\big(\svHpcoef^{k}_{i}+\frac{\rho_{i}}{2}\big)\big\Vert \svHp^{k+1}_{i}-\svHp^{k}_{i}\big\Vert^2
+\big(\svHncoef^{k}_{i}+\frac{\rho_{i}}{2}\big)\big\Vert \svHn^{k+1}_{i}-\svHn^{k}_{i} \big\Vert^2
\notag\\&
+\frac{\rho_{i}}{2}\sum_{l}
  \Big(              \big\Vert \pcF_{i}(X^{k+1,k+1}_{i,l})-\pcF_{i}(X^{k+1,k}_{i,l}) \big\Vert^2
                    +\big\Vert\pcG_{i,l}(X^{k+1}_{i,l}-X^{k}_{i,l}) \big\Vert^2
                    +2\big\Vert \pcH_{i,l}(X^{k+1}_{i,l}-X^{k}_{i,l}) \big\Vert^2
    \Big),
\label{GlobalInequalityPkiDefn}
\end{align}
and
\begin{align}
 U^{k}_{i,l}=
 \frac{1}{2}\big\langle \dvF^{k}_{i}+\rho_i\big(\pcF_{i}(X^{k+1,k+1}_{i,l})+\svF^{k}_{i}\big),
                (X^{k}_{i,l}-X^{k+1}_{i,l})^{T}\,
        \nabla_{X_{i,l}}\nabla_{X_{i,l}}\pcF_{i}(X^{k+1,k+1}_{i,l})\,
              (X^{k}_{i,l}-X^{k+1}_{i,l})
            \big\rangle.
\label{GlobalInequalityUklDefn}
\end{align}
The global inequality~\eqref{GlobalInequalitydvSum2infty}
provides a basis for the discussion,
as a strategy of descent adopted by the algorithm.

We address the issues of initialization and parameter choices mentioned
in Assumption~\ref{AssumptionSetII}(A).
Inequality~\eqref{GlobalInequalitydvSum2infty} indicates that
the primal and dual sequences need to be initialized appropriately
so that $L^{0}_{i}$ is
great to have \eqref{GlobalInequalitydvSum2infty} hold.
Specifically, $L^{0}_{i}$ needs to be much greater than the optimal value $f(Z^{\infty})$
and each and every sub-sum $\sum_{k=1}^{\infty}D^{k}_{i,j}$ on the right-hand side
of \eqref{GlobalInequalitydvSum2infty} contains an infinite number of
nonzero elements, which is essentially
equivalent to that all the dual component
sequences are not trapped,
as demonstrated by Proof~\ref{ProofOfPDDeltaLimitsLemma}(A).
To this end,
one set of possible choices for the parameter and initial values is given as follows.

(A)
The upper bounds on the slack variables are taken as
\begin{align}
\usvZp_{Y_{i}}=\usvZn_{Y_{i}}=2,\
\usvF_{Y_{i}}=2,\
\usvG_{Y_{i}}=2,\
\usvHp_{Y_{i}}=\usvHn_{Y_{i}}=2.
\label{UpperBoundsSlackVariables}
\end{align}
These choices are made on the basis of Assumption~\ref{AssumptionSetI},
and to guarantee Assumption~\ref{AssumptionSetII}(B).
Based on the forms of \eqref{SolvingsvZpk}
through \eqref{SolvingsvHnk},
the upper bounds on the dual variables are taken as
\begin{align}
\big\{\udvZp_{\mu_{i}},\udvZn_{\mu_{i}},\udvF_{\mu_{i}},\udvG_{\mu_{i}},
\udvHp_{\mu_{i}},\udvHn_{\mu_{i}}\big\}
\propto \rho_{i}
\big\{\usvZp_{Y_{i}},\usvZn_{Y_{i}},\usvF_{Y_{i}},\usvG_{Y_{i}},\usvHp_{Y_{i}},
        \usvHn_{Y_{i}}\big\},
\tag{\ref{UpperBoundsDualVariables}}
\end{align}
where the dual upper bound proportionality coefficients take a value of 5, say,
to have the bounds great, helping make the dual component sequences
not trapped in the NUBs.

(B)
$Z^{0}_{l,j}=\lambda_{Z}\, \text{sign}(f_{l,j})\, u_{l,j}$,
$\lambda_{Z}\in[0,0.9]$, say, to have $Z^{0}$ well
inside the interior of $[-u,u]$, under the proximal parameter values
for $\{\proxXpropto_{i,l},\proxZpropto\}$
specified by \eqref{Norm2ProptCoeffsDvControlParameters}.
Then,
$f(Z^{0})=\lambda_{Z}\sum_{l,j}\vert f_{l,j}\vert u_{l,j}$;
$X^{0}_{i}=Z^{0}$ for all $i$,
consistent with the limits of $X^{\infty}_{i}=Z^{\infty}$ for all $i$.
A greater $\lambda_{Z}$ provides a greater $L^{0}$.

(C)
The initial values for the slack and dual sequences are assigned through
\begin{align}
\svZp^{0}_{i}=\svZn^{0}_{i}=2,\
\svF^{0}_{i}=1,\
\svG^{0}_{i}=1,\
\svHp^{0}_{i}=\svHn^{0}_{i}=2,
\label{InitializationSlackSequences}
\end{align}
and
\begin{align}
  &
\dvZp^{0}_{i}=\rho_{i}\svZp^{0}_{i},\
\dvZn^{0}_{i}=\rho_{i}\svZn^{0}_{i},\
\dvF^{0}_{i}=10 \eF^{0}_{i},\
\dvG^{0}_{i}=10 \eG^{0}_{i},\
\dvHp^{0}_{i}=\rho_{i}\eHp^{0}_{i},\
\dvHn^{0}_{i}=\rho_{i}\eHn^{0}_{i},
\label{InitializationDualVariables}
\end{align}
in order to make $\{\eZp^{0}_{i,j}$,\,$\eZn^{0}_{i,j}$,\,$\eF^{0}_{i,j}$,\,$\eG^{0}_{i,j}$,\,$\eHp^{0}_{i,j}$,\,$\eHn^{0}_{i,j}\}$
positive and $L^{0}$ great.
It is computationally permissible that the initial values
in \eqref{InitializationSlackSequences}
are equal to the upper bounds, considering that
the slack component sequences decrease monotonically
during the interval of initial iterations.
 Also, adequately great initial values for the slack sequences
are required, owing to their impacts on
the signs of the extended component sequences
and the number of iterations toward
the desired zero limits of the extended sequences.
Together with the parameter values,
such a great number of iterations correspond to slower paces
of the evolutions of the extended and dual component sequences,
providing effective interactions among the component sequences
via the update rules so as to realize Assumption~\ref{AssumptionSetII}(A).
The relatively low values of $\{\svF^{0}_{i},\svG^{0}_{i},\dvF^{0}_{i},\dvG^{0}_{i}\}$
are chosen, taking into account that the vast majority of
$\{(\pcF_{i}(X_{i}))_{j}\leq 0$, $(\pcG_{i}(X_{i}))_{j}\leq 0: j\}$
are not active and the dual limits themselves satisfy
the stationarity condition~(\ref{StationarityConditionsThm1}$)_1$,
unlike the differences of $\{\dvZp^{\infty}_{i,l}-\dvZn^{\infty}_{i,l},
      \dvHp^{\infty}_{i}-\dvHn^{\infty}_{i}\}$.
The choice of $\{\dvF^{0}_{i}=0$, $\dvG^{0}_{i}=0\}$ is not taken,
for the associated dual sequences to play a more active role.
The proposed values help to generate great $L^{0}_{i}$ under $X^{0}_{i}=Z^{0}$,
\begin{align}
L^{0}_{i}=\,&
 f(Z^{0})
+\frac{3\rho_{i}}{2}\Vert \svZp^{0}_{i} \Vert^2
+\frac{3\rho_{i}}{2}\Vert \svZn^{0}_{i} \Vert^2
+\big(10+\frac{\rho_{i}}{2}\big)\Vert \eF^{0}_{i} \Vert^2
  +\big(10+\frac{\rho_{i}}{2}\big)\Vert \eG^{0}_{i} \Vert^2
  +\frac{3\rho_{i}}{2}\Vert \eHp^{0}_{i} \Vert^2
  +\frac{3\rho_{i}}{2}\Vert \eHn^{0}_{i} \Vert^2.
\label{Li0}
\end{align}
This expression illustrates one important role played by the slack variables,
especially
$\{\svZp_{i},\svZn_{i},\svHp_{i},\svHn_{i}\}$
introduced via the extended constraints of equality in that 
the choices of $\{\svZp^{0}_{i}, \svZn^{0}_{i}, \svHp^{0}_{i}, \svHn^{0}_{i}\}$
and $\{\dvZp^{0}_{i}, \dvZn^{0}_{i}, \dvHp^{0}_{i}, \dvHn^{0}_{i}\}$
can make $L^{0}_{i}$ great under $X^{0}_{i}=Z^{0}$.

(D)
The following choices of values for the dual scalar coefficients
of \eqref{DualUpdateRulesdvEFGBoundsj},
\begin{align}
\big\{
\dvZpcoef_{i},\dvZncoef_{i},\dvFcoef_{i},\dvGcoef_{i},
\dvHpcoef_{i},\dvHncoef_{i}\big\}=10^{-2},
\label{DualControlParametersEstimated}
\end{align}
intend to help uphold the global inequality~\eqref{GlobalInequalitydvSum2infty}
and to make each and every sub-sum $\sum_{k=1}^{\infty}D^{k}_{i,j}$
of \eqref{GlobalInequalitydvSum2infty} contain an infinite number of nonzero elements.

(E)
We tentatively take
\begin{align}
 \rho_{i}\in[10^{2}, 10^{3}].
 \tag{\ref{PenaltyParameter}}
\end{align}
Though absent in the stationarity conditions~\eqref{StationarityConditionsThm1},
the value of $\rho_{i}$ is treated as a reference to fix
some other parameter and initial values in the above;
together with \eqref{Norm2ProptCoeffsDvControlParameters}, \eqref{InitializationSlackSequences}.
and \eqref{InitializationDualVariables}, it controls
the evolution of the extended component sequences,
especially at great $k$,
as characterized by \eqref{EffectsOfGreaterPenalty}.
How to assign a specific value to $\rho_{i}$
or even its range is yet to be resolved.
The present formulation of the algorithm has one characteristic:
the great disparity between the scales of the constraints
specified by Assumption~\ref{AssumptionSetI}(A)(B)(C), $O(1)$,
and the scales of the 2-norm proximal control parameters, greater than
$O(10^5)$ or $O(10^6)$ from \eqref{Norm2ProptCoeffsDvControlParameters}
and \eqref{PenaltyParameter} along with
\eqref{ProximalControlParameter2sigmaA}
through \eqref{ProximalControlParameternHgammaA}.

\section{Summary}\label{sec:Summary}
 
 The present study concerns distributed computing for huge-scale
 aggregative convex programming of a linear objective
 subject to the affine constraints of equality and inequality and
 the quadratic constraints of inequality that are convex
 and aggregatively computable.
It develops an algorithm composed of
\eqref{ACPGSProximalXilA},
\eqref{ACPGSProximalZl},
\eqref{ACPGSProximalsvZpiA} through \eqref{ACPGSProximalsvHniA},
\eqref{DualUpdateRules}, \eqref{DualUpdateRulesdvEFGBoundsj}, \eqref{proxXNormOcoefkilDefn},
 and \eqref{ProximalControlParameter2sigmaA}
through \eqref{ProximalControlParameternHgammaA},
adopting the global consensus with single common variable,
the extended constraints of equality, subblocks,
augmented Lagrangian, and proximal point algorithm.
The global consensus is used to partition the constraints of equality and
inequality into $N$ multi-consensus-blocks,
 and the subblocks of each consensus block
 are used to partition the primal variables
 into $M$ sets of disjoint subvectors.
 The global consensus constraints of equality
 and the original constraints are converted
 into the extended constraints of equality via slack variables
 in order to help resolve feasibility and initialization of the algorithm.
  The proximal point method with double proximal terms or single,
 the block-coordinate Gauss-Seidel method,
and ADMM are used to update the primal (and slack) sequences;
 the descent models with built-in bounds are used to update the dual,
  motivated by the mathematical structures of the first-order characteristics
 of the update rules for the primal (and slack) sequences.
 The feasibility conditions for the algorithm to produce
 optimal solutions are described and their realizations
 through precondition and parameter and initial values are outlined or specified.
 Under the feasibility conditions supposed,
 convergence of the algorithm to optimal solutions
 is argued and the rate of convergence, $O(1/k^{1/2})$ is estimated roughly.

A few issues need to be explored further in a comprehensive
and quantitative manner analytically and computationally,
such as the feasibility conditions
introduced in Assumption~\ref{AssumptionSetII}(A),
 the set of choices for the parameter and initial values
presented in Secs.~\ref{sec:UpdateDual} and \ref{sec:Initialization},
and more.

(A)
The value of $N$ and the possible redundant use of
  $\{\pcG(Z)\leq 0$, $\pcH(Z)=0\}$ need to be studied,
  concerning the computational size of the algorithm
  against the possibility of $\proxXNormOcoef^{k}_{i,l}=0$ for all $k$
  which makes it simpler to update $X^{k+1}_{i,l}$.
Also, the value of $N$ is relevant to the resolution of the issue raised in
Assumption~\ref{AssumptionSetI}(D).
The other extreme is to have $\proxXNormOcoef^{k}_{i,l}>0$ for all $k$.

(B)
The penalty parameter $\rho_{i}$ is used as a reference to assign
values to some parameters and the initialization of certain dual sequences.
Considering its role in the evolution of $\{L^{k}\}_{k}$
and the feasibility conditions,
its value has to be specified and tested appropriately.

(C)
The values are proposed tentatively for the dual scalar coefficients,
$\{\dvZpcoef_{i}$,\,$\dvZncoef_{i}$,\,$\dvFcoef_{i}$,\,$\dvGcoef_{i}$,\,$\dvHpcoef_{i}$,\,$\dvHncoef_{i}\}$
of \eqref{DualUpdateRulesdvEFGBoundsj},
the 2-norm proximal proportionality coefficients,
$\{\proxXpropto_{i,l}$,\,$\proxZpropto$,\,$\proxsvpXpropto_{i}$,\,$\proxsvnXpropto_{i}$,\,$\proxsvFpropto_{i}$,\,$\proxsvGpropto_{i}$,\,$\proxsvpHpropto_{i}$,\,$\proxsvnHpropto_{i}$,\,$\proxXproptorhoi$,\,$\proxZproptorhoi$,\,$\proxsvpXproptorhoi$,\,$\proxsvnXproptorhoi$,\,$\proxsvFproptorhoi$,\,$\proxsvGproptorhoi$,\,$\proxsvpHproptorhoi$,
$\proxsvnHproptorhoi\}$ of \eqref{ProximalControlParameter2sigmaA}
through \eqref{ProximalControlParameternHgammaA},
the slack and dual bounds, and the initialization.
They are to be tested computationally and adjusted,
regarding the balance and compatibility among these values
in order to realize the feasibility conditions
and optimize the evolution paces of the sequences.

(D)
An efficient scheme is required to solve $X^{k+1}_{i,l}$
from \eqref{ACPGSProximalXilA}.

(E)
Possible calibration and application of the proposed algorithm
to small-to-large scale problems are yet to be done.

(F)
Specific implementation, calibration, and tests of the algorithm
are yet to be carried out, possibly with the aid of AI.

(G)
An improved analysis of the rate of convergence of the algorithm
is needed.

(H)
There is the need to consider aggregative convex programming of more general forms like
$\{f(S), H(S,T)=0, H'(T,V)=0$,
$G(S)\leq 0$, $G'(V)\leq 0$, $G''(S,V)\leq 0$, $G'''(S,T,V)\leq 0$,
$F(S)\leq 0$, $F'(V)\leq 0$, $F''(S,V)\leq 0$, $F'''(S,T,V)\leq 0\}$
where the dimension of $S$ is much smaller than that of $T$
and the latter smaller than that of $V$.
We may apply the present algorithm by redefining $\{X_i,Z,u,\ldots\}$ accordingly.

(I)
Distributed computing for huge-scale linear programming is treated
as a special case of the present without $\pcF(Z)\leq 0$,
updating the work of \cite{Tao2025v7LP}.

(J)
Explore a possible extension of the framework
to cover the quadratic and aggregatively computable
 constraints of inequality involving difference of convex functions (DC), like
$e(Z)=a(Z)+b(Z)^2-c(Z)^2\leq 0$
where $a(Z)$, $b(Z)$, and $c(Z)$ are linear or affine,
with the help of great values of $\proxXNormTcoef^{k}_{i,l}$ to make
the function of \eqref{ACPGSProximalXilA} convex, etc.

\begin{acknowledgments}
I thank Professor M. Ramakrishna in the Department of Aerospace Engineering at IIT Madras
 for discussions.
\end{acknowledgments}

\bibliography{DCHSACP.bib}

@PREAMBLE{
 "\providecommand{\noopsort}[1]{}" 
# "\providecommand{\singleletter}[1]{#1}%" }

@BOOK{Bertsekas1996,
   author       = {Dimitri P. Bertsekas},
   year         = "1996",
   title        = {Constrained Optimization and Lagrange Multiplier Methods},
   publisher    = {Athena Scientific, Belmont, Massachusettes},
}

@MISC{Boydetal2010,
   author       = "S. Boyd and N. Parikh and E. Chu and B. Peleato and J. Eckstein",
   title        = "Distributed optimization and statistical learning via the alternating direction method of multipliers", 
   howpublished = "Foundations and Trends in Machine Learning. \textbf{3}, 1-122",
   year         = "2010", 
}

@MISC{Mangasarian1984,
   author       = "O. L. Mangasarian",
   title        = "Some applications of penalty functions in mathematical programming", 
   howpublished = "MRC Technical Summary Report \#2720, Mathematics Research Center, University of Wisconsin-Madison",
   year         = "1984", 
}

@MISC{ParikhBoyd2013,
   author       = "N. Parikh and S. Boyd",
   title        = "Proximal Algorithms", 
   howpublished = "Foundations and Trends in Optimization. \textbf{1}, 123-231",
   year         = "2013", 
}

@MISC{SunSun2024,
   author       = "Kaizhao Sun and Xu Andy Sun",
   title        = "Dual descent augmented Lagrangian method and alternating direction method of multipliers", 
   howpublished = "SIAM J. Optim. \textbf{34}, 1679-1707",
   year         = "2024", 
}

@MISC{Tao2020,
   author       = "Luoyi Tao",
   title        = "Homogeneous shear turbulence as a second-order cone program", 
   howpublished = "e-print arXiv:1408.0376v6 [physics.flu-dyn]",
   year         = "2020",
}

@MISC{Tao2025v7LP,
   author       = "Luoyi Tao",
   title        = "Distributed computing for huge-scale linear programming", 
   howpublished = "e-print arXiv:2408.06204v7 [math.OC]",
   year         = "2025",
}

\end{document}